\documentclass[reqno]{amsart}
\pdfoutput=1

\usepackage[all,cmtip,line]{xy}

\usepackage[margin=1in]{geometry}
\usepackage{tikz}

\def\arXiv#1{arXiv:\href{http://arXiv.org/abs/#1}{#1}}

\usepackage{amsmath, amstext, amssymb, amsthm, amscd}
\usepackage{microtype}

\usepackage[colorlinks,bookmarks=false]{hyperref}
\hypersetup{citecolor=blue}

\newtheorem{theorem}{Theorem}

\newtheorem{proposition}[theorem]{Proposition}

\theoremstyle{definition}

\newtheorem{remark}[theorem]{Remark}

\newcommand{\R}{{\mathbb R}}
\newcommand{\C}{{\mathbb C}}
\newcommand{\Z}{{\mathbb Z}}
\newcommand{\Q}{{\mathbb Q}}

\newcommand{\A}{{\mathcal A}}
\newcommand{\Proj}{{\mathbb P}}
\newcommand{\sO}{{\mathcal O}}

\newcommand{\NS}{\mathrm{NS}}

\newcommand{\End}{\mathrm{End}}
\newcommand{\sC}{\mathcal C}
\newcommand{\M}{\mathcal M}
\newcommand{\sL}{\mathcal L}
\newcommand{\sF}{{\mathcal F}}

\newcommand{\sH}{\mathcal{H}}

\newcommand{\I}{\mathrm{I}}
\newcommand{\II}{\mathrm{II}}
\newcommand{\IV}{\mathrm{IV}}
\newcommand{\III}{\mathrm{III}}
\newcommand{\eps}{\epsilon}

\newcommand{\MoW}{Mordell-\mbox{\kern-.12em}Weil}

\newcommand{\kbar}{{\kern.06ex\overline{\kern-.06ex k}}}
\newcommand{\Qbar}{{\kern.1ex\overline{\kern-.1ex\Q\kern-.1ex}\kern.1ex}}

\newcommand{\SL}{\mathop{\rm SL}\nolimits}
\newcommand{\Sp}{\mathop{\rm Sp}\nolimits}


\title[Elliptic subfields of genus $2$ function fields]{Hilbert
  modular surfaces for square discriminants \\ and elliptic subfields
  of genus $2$ function fields}

\author{Abhinav Kumar}
\address{Department of Mathematics\\
Massachusetts Institute of Technology\\
Cambridge, MA 02139}
\curraddr{Department of Mathematics\\
Stony Brook University\\
Stony Brook, NY 11794}
\email{thenav@gmail.com}

\date{September 26, 2016}

\begin{document}

\begin{abstract}
We compute explicit rational models for some Hilbert modular surfaces
corresponding to square discriminants, by connecting them to moduli
spaces of elliptic K3 surfaces. Since they parametrize decomposable
principally polarized abelian surfaces, they are also moduli spaces
for genus-$2$ curves covering elliptic curves via a map of fixed
degree. We thereby extend classical work of Jacobi, Hermite, Bolza
etc., and more recent work of Kuhn, Frey, Kani, Shaska, V\"olklein,
Magaard and others, producing explicit families of reducible
Jacobians. In particular, we produce a birational model for the moduli
space of pairs $(C,E)$ of a genus $2$ curve $C$ and elliptic curve $E$
with a map of degree $n$ from $C$ to $E$, as well as a tautological
family over the base, for $2 \leq n \leq 11$. We also analyze the
resulting models from the point of view of arithmetic geometry, and
produce several interesting curves on them.
\end{abstract}

\maketitle

\section{Introduction}

In algebraic geometry and number theory, one is frequently interested
in abelian varieties with extra endomorphisms, and their moduli
spaces. The study of elliptic curves and their moduli spaces has been
extremely influential in the last century (for instance, see \cite{Ma,
  GZ, Wi}).  Higher dimensional abelian varieties have also seen many
applications; however, it has remained quite a challenge to provide
explicit computational descriptions for these or their moduli
spaces. There has been quite a lot of work in the direction of
explicit approaches to abelian surfaces in the last couple of
decades. This paper is chiefly concerned with moduli spaces of genus
$2$ curves with decomposable Jacobians. We say that an abelian variety
$A$ over a field $k$ is decomposable or reducible over $k$ if it is
isogenous over $k$ to a product of abelian varieties of smaller
dimension.

In particular, we will be interested in principally polarized abelian
surfaces $A$ which are isogenous to a product of elliptic curves. A
natural way to produce such an abelian surface is as the Jacobian of a
curve $C$ of genus $2$, which has a map of degree $n > 1$ to an
elliptic curve $E$. The morphism $\phi \colon C \to E$ induces $\phi^*
\colon E \to J(C) = A$ and $\phi_* \colon A = J(C) \to E$, with
$\phi_* \circ \phi^* = n_E$. The curve $C$ also has a map of degree
$n$ to the complement $E' = A/\phi^*(E)$ (obtained by composing the
embedding $C \to A$ with the projection map) and $A$ is isogenous to
the product $E \times E'$ with kernel isomorphic to $\Z/n\Z \oplus
\Z/n\Z$. We say that $J(C)$ is $(n,n)$-split. The (coarse) moduli
space we wish to parametrize is the space $\widetilde{\sL}_n$ of pairs
$(C,E)$ related by an {\em optimal} map of degree $n > 1$. (Here,
optimal means that the map does not factor through an unramified cover
of $E$.) It is a double cover of the moduli space $\sL_n$ of genus $2$
curves $C$ whose function field has an elliptic subfield of degree
$n$; this latter locus is a hypersurface in $\mathcal{M}_2$, the
moduli space of genus $2$ curves.

Classically, hyperelliptic curves with split Jacobians were studied in
the guise of the reduction of abelian integrals to elliptic integrals
by algebraic transformations, going as far back as Legendre and
Jacobi, who essentially gave a complete description of the moduli
space for $n = 2$. The generic family of curves $C$ for $n = 3$ was
worked out by Hermite, Goursat, Burkhardt, Brioschi and Bolza, and $n
= 4$ by Bolza. We refer the reader to \cite{Kr} for a summary of this
classical literature. Kuhn \cite{Kuh} revisited this topic, giving a
combinatorial description in terms of the branch points of the map
$\phi$ as well as the hyperelliptic involution $\iota$ on $C$, and
adapting some examples to positive characteristic. This approach has
led to fairly explicit descriptions of $\widetilde{\sL}_n$ for $2 \leq
n \leq 5$ by Shaska and others \cite{SV, Sh, SWWW, MSV}. From an
arithmetic standpoint, reducible Jacobians for $2 \leq n \leq 4$ have
been studied in \cite{CF, BD, BFT}. Frey and Kani \cite{FK, Fr} have
also studied such covers for general $n$ by looking at pairs of
elliptic curves $E$ and $E'$ such that there is an Galois equivariant
anti-isometry between their $n$-torsion subschemes. Their description
connects $\widetilde{\sL}_n$ to several interesting arithmetic
questions and applications.

In this paper, we will take a rather different strategy to compute the
moduli spaces $\widetilde{\sL}_n$, which does not focus on the map
$\phi \colon C \to E$ and its ramification locus. The key fact upon which
our method rests is that a genus $2$ curve $C$ has a map of degree $n$
to an elliptic curve $E$ if and only if the point in $\A_2$
corresponding to its Jacobian $J(C)$ lies on the Humbert surface
$\sH_{n^2}$ (see \cite{Mu}, for instance). In other words, under the
birational morphism $\mathcal{M}_2 \to \A_2$, the image of $\sL_n$ is
(birationally) identified with the Humbert surface for discriminant
$n^2$. Recall that for a discriminant $D$, the Humbert surface $\sH_D$
describes principally polarized abelian surfaces which have real
multiplication by the quadratic ring of discriminant $D$. Its double
cover is the Hilbert modular surface $Y_{-}(D)$ which parametrizes
pairs $(A, \iota)$, where $A$ is a principally polarized abelian
surface, and $\iota \colon \sO_D \to \textrm{End}(A)$ is a ring
homomorphism. Note that the reason $Y_{-}(D) \to \sH_D$ is a double
cover is because (generically) there are two choices for the action of
$\sqrt{D}$.

In \cite{EK}, a method was laid out to compute explicit models for
these Hilbert modular surfaces $Y_{-}(D)$, relying on the computation
of moduli spaces of suitable elliptic K3 surfaces, which are related
to the abelian surfaces via a Shioda--Inose structure. In that paper,
equations were given for rational models for the surfaces $Y_{-}(D)$,
for the thirty fundamental discriminants $D$ such that $1 < D <
100$. These Hilbert modular surfaces are coarse moduli spaces for
principally polarized abelian surfaces with real multiplication by the
ring of integers $\sO_D$ of $\Q(\sqrt{D})$. 

In this paper, we will extend this method to compute the Humbert
surfaces for {\em square} discriminants $D = n^2$, and our desired
moduli space $\widetilde{\sL}_n$ (which is birational to
$Y_{-}(n^2)$), for $2 \leq n \leq 11$. To do so requires several ideas
beyond those of \cite{EK}. First, we generalize the set-up and
theorems in \cite{EK} to the case of non-fundamental
discriminants. This allows us to compute $Y_{-}(n^2)$ as a moduli
space of K3 surfaces, but does not elucidate its structure as a moduli
space of $(C,E)$. In order to do so, we first describe explicitly a
normalized tautological family of genus $2$ curves over the moduli
space, using the method of certifying eigenforms for real
multiplication from \cite{KM1}. Finally, we describe a new method
involving integration of eigenforms to recover the elliptic curves
$E_1$ and $E_2$ associated to a given $C$.

Our models agree with those in the literature for $2 \leq n \leq 5$,
and we go well beyond the previous state of the art with the higher
values of $n$. In doing so, we exhibit non-rational surfaces
$\widetilde{\sL}_n$: the surfaces $\widetilde{\sL}_6$ and
$\widetilde{\sL}_7$ are elliptic K3, the next three
$\widetilde{\sL}_8, \widetilde{\sL}_9$ and $\widetilde{\sL}_{10}$ are
honestly elliptic, whereas $\widetilde{\sL}_{11}$ is a surface of
general type. The geometric classification of these surfaces was
studied in \cite{He, KS}, but explicit algebraic models were not known
before this work.  We anticipate that our algebraic models for these
moduli spaces will make it easy to carry out explicit arithmetic and
geometric investigations. For instance, one can immediately apply our
formulas to produce examples or even families of elliptic curves with
anti-isometric $n$-torsion Galois representations.

The plan of the paper is as follows: in Section \ref{setup}, we prove
the necessary generalizations of the results from \cite{EK}. Each of
the following sections treats a separate square discriminant. We
outline the parametrization of the moduli space of elliptic K3
surfaces, and a sequence of ``elliptic hops'' converting to an
elliptic fibration with $E_8$ and $E_7$ fibers. This allows us to
compute the map to $\A_2$. We can then write down the explicit double
cover giving a model of $\widetilde{\sL}_n = Y_{-}(n^2)$. We carry out
some basic geometric analysis of the moduli space, and note any
obvious arithmetic curves on it. We also provide a tautological family
of genus $2$ curves over $\widetilde{\sL}_n$, as well as the
$j$-invariants of both the associated elliptic curves\footnote{We can
  make a uniform choice of one of these, so that the family of
  $(C,E_1)$ is really the tautological family over the moduli space
  $\widetilde{\sL}_n$.} $E_1$ and $E_2$ .

The auxiliary computer files for this paper, containing various
formulas and computations omitted here for lack of space, are
available from \url{http://arxiv.org/abs/1412.2849}. To access these,
download the source file for the paper. This will produce not only the
\LaTeX{} file for this paper, but also the computer code.

For the computations involved in our work, we made frequent use of the
computer algebra systems \texttt{pari-gp}, \texttt{Magma},
\texttt{Maxima} and \texttt{Sage}. We thank Nils Bruin, Henry Cohn,
Noam Elkies, Gerard van der Geer, David Gruenewald, Curt McMullen,
Ronen Mukamel, Dan Petersen, Matthias Sch\"utt and the anonymous
referee for helpful comments. This work was supported in part by NSF
grant DMS-0952486, and by a grant from the MIT Solomon Buchsbaum
Research Fund.

\section{Hilbert modular surfaces for square discriminants} \label{setup}

In this section, we generalize the main theorems of \cite{EK} to the
case of non-fundamental discriminant $D$, which is a positive integer
congruent to $0$ or $1$ modulo $4$. In particular, this paper will
deal with the case of square discriminant: $D = n^2$, in which case
the ring in question is an order in a split quadratic algebra
(i.e.~$\Q \oplus \Q$), rather than a real quadratic field. We first
briefly go over the basic setup for this particular case, since the
typical case considered in the literature is that of non-square
discriminant. For more background on Hilbert modular surfaces see
\cite{vdG}, and for a quick introduction to the calculations of this
section, we refer the reader to \cite{Mc}. The notes \cite{Do} also
contain an excellent description of the theory of abelian varieties
with extra endomorphisms, and the connection with K3 surfaces which
was exploited in \cite{Kum1, EK} and is also used here.

\subsection{The quadratic ring of discriminant $D$}
Let $D = n^2$, and $\sO_D$ the quadratic ring of discriminant
$D$. Concretely, we can write
\[
\sO_D = \{ (x,y) \in \Z^2 : x \equiv y \bmod n \}.
\]
with componentwise addition and multiplication. It has a $\Z$-basis
$\{e_1, e_2\}$ given by $e_1 = 1 = (1,1)$ and $e_2 = (n,0)$.  The two
embeddings $\sigma_1,\sigma_2$ of $K = \sO_D \otimes \Q \cong \Q
\oplus \Q$ into $\R$ are given by taking the first or second
components and composing with $\Q \hookrightarrow \R$. There are four
square roots of $D$ in $\sO_D$, namely $\pm n = \pm (n,n)$ and $\pm
(n, -n)$. If we label $\sqrt{D} = (n,-n)$, we see that $1$ and $(D +
\sqrt{D})/2 = \big( (n^2 +n)/2, (n^2 - n)/2 \big)$ also form a
$\Z$-basis, just as in the case of non-square discriminant. The
involution which changes $\sqrt{D}$ to $-\sqrt{D}$ is that which
switches the two factors; i.e.~takes $(x,y)$ to $(y,x)$. The trace map
takes $(x,y)$ to $x + y$, and the matrix of the trace form with
respect to $e_1$ and $e_2$ is
\[
\begin{pmatrix}
2 & n \\ n & n^2
\end{pmatrix},
\]
which has discriminant $n^2 = D$. The dual lattice $\sO_D^*$ with
respect to the trace form has basis $f_1 = (1,0)$ and $f_2 =
\frac{1}{n}(1,-1)$. (In fact, for a general discriminant, $\sO_D^*$
equals $\frac{1}{\sqrt{D}} \sO_D$.) The discriminant group
$\sO_D^*/\sO_D$ is easily calculated to be $\Z/n^2\Z$ if $n$ is odd,
and $\Z/2\Z \oplus \Z/(n^2/2)\Z$ if $n$ is even.

\subsection{The setup for non-fundamental discriminants}

We now adapt Sections $3$ and $4$ of \cite{EK} to the setting of
non-fundamental discriminant. The outline of the method is as follows.

\begin{enumerate} 
\item The Hilbert modular surface $Y_{-}(D)$ is the coarse moduli
  space of $(A, \iota)$ where $A$ is a principally polarized abelian
  surface, and $\iota \colon \sO_D \to \End(A)$ is a homomorphism. The
  complex manifold $Y_{-}(D)_\C$ can be obtained by compactifying
  $\SL_2(\sO_D, \sO_D^*) \backslash \sH^2$.
\item There is a map $Y_{-}(D) \to \A_2$ which simply takes the ppas
  $A$. The image $\sH_D$ is the Humbert surface of discriminant $D$ in
  $\A_2$, and the map from the Hilbert modular surface has degree
  $2$. The key first step is to realize $\sH_D$ as a moduli space of
  elliptic K3 surfaces, polarized by a particular lattice $L_D$ of
  rank $18$ and discriminant $-D$. We explicitly parametrize this
  moduli space $\M_{L_D}$, and compute the universal K3 family over
  it.
\item We then find a different elliptic fibration on the (family of)
  K3 surfaces, with reducible fibers of type $\II^*$ and
  $\III^*$. This is accomplished by a sequence of ``elliptic hops'',
  which reflect $2$- and $3$-neighbor steps at the lattice level (see
  \cite[Section 5]{EK} or \cite[Appendix A]{Kum2}).  By the main
  result of \cite{Kum1}, from the final Weierstrass equation we may
  read out the map $\M_{L_D} \to \A_2$, which is given in terms of
  Igusa--Clebsch invariants (coordinates on $\A_2$ \cite{Ig}).
\item So far, we have an explicit birational model of $\sH_D$. To get
  to $Y_{-}(D)$, we must identify the correct double cover. We first
  pin down the geometric branch locus as a union of certain modular
  curves (by a result of Hausmann), and interpret these curves in
  terms of the K3 moduli space (the Picard group jumps in a
  predictable manner). Finally, we find the correct arithmetic twist
  by point-counting and matching characteristic polynomials of
  Frobenius on the abelian and K3 sides. 
\end{enumerate}

The first part of \cite[Section $3$]{EK} outlines the proof that
$\SL_2(\sO_D, \sO_D^*) \backslash \sH^2$ is the coarse moduli space of
principally polarized abelian surfaces with real multiplication by
$\sO_D$. The same proof goes through for arbitrary discriminant. Of
course, for $D = n^2$, usually it is more convenient to understand a
ppas with real multiplication by $\sO_D$ in terms of pairs of elliptic
curves which have degree $n$ maps from a common genus $2$ curve, as
described in the introduction. This is the approach taken in \cite{He,
  KS, Ca}.

\begin{proposition}
Let $A$ be a principally polarized abelian surface with $\End(A) \cong
\sO_D$.  Then $\NS(A) \cong \End(A)$.  The lattice $\NS(A)$ has a
basis with Gram matrix
\begin{equation}
\left(\begin{array}{cc}
\label{eq:ODGram}
2 & D \\
D & (D^2 - D)/2
\end{array}\right)
\end{equation}
of signature $(1,1)$ and discriminant $-D$.
\end{proposition}

The proof of \cite{EK} goes through with minor changes, noting that
the Rosati involution on $K$ must be the identity, since the other
element of the Galois group (interchange of the factors in the
square-discriminant case) is not a positive involution.

The proof of the next proposition is unchanged from the original.

\begin{proposition}
  There is a primitive embedding, unique up to isomorphism, of the lattice
  $\sO_D$ into $U^3$.  Let $T_D$ be the orthogonal complement of
  $\sO_D$ in $U^3$.  Then there is a primitive embedding,
  unique up to isomorphism, of $T_D$ into the K3 lattice $\Lambda$.
\end{proposition}

The main theorem of that section identifies the Humbert surface with a
moduli space of K3 surfaces lattice-polarized by $L_D := E_8(-1)^2
\oplus \sO_D$. Let $L$ be the lattice $U \oplus E_8(-1) \oplus
E_7(-1)$.

\begin{theorem} \label{humbert}
  Let $\sF_{L_D}$ be the moduli space of K3 surfaces that are lattice
  polarized by $L_D$.  Then the isomorphism $\phi \colon \sF_L \rightarrow
  \A_2$ of \cite{Kum1} induces a birational surjective morphism
  $\sF_{L_D} \rightarrow \sH_D$.
\end{theorem}

Its proof relies on a key proposition involving the lattice embedding
$L \hookrightarrow L_D$, where $L = U \oplus E_8(-1) \oplus E_7(-1)$
and $L_D \cong E_8(-1)^2 \oplus \sO_D$ is the orthogonal complement of
$T_D$ in $\Lambda$, uniquely described by the proposition above. The
proof of the key proposition and the main theorem go through without
any changes.

We must next show that the branch locus of the map $Y_{-}(D) \to
\sH_D$ corresponds to a divisor $Z$ in the moduli space of $\M_{L_D}$
of K3 surfaces lattice polarized by $L_D$, such that the rank of the
K3 surface corresponding to a generic point on $Z$ jumps to $19$, and
the discriminant of the Picard group is $D/2$ or $2D$. In \cite{EK},
the argument used the fact that $D$ was fundamental; we give a more
general proof here, valid for arbitrary discriminant.

First, we observe as in \cite{EK} that the branch locus consists of a
union of specific modular curves. Let us consider the birational
models $\Gamma \backslash \sH^2$ for $Y_{-}(D)$ (where $\Gamma =
\SL_2(\sO_D, \sO_D^*)$) and $\Sp_4(\Z) \backslash \mathcal{S}_2$ for
$\A_2$, with $\mathcal{S}_2$ being the Siegel upper half space for
genus $2$. The two-to-one map $\Gamma \backslash \sH^2 \to \Sp_4(\Z)
\backslash \mathcal{S}_2$ factors through $(\Gamma \cup \Gamma \sigma
) \backslash \sH^2$, where $\sigma$ is the interchange of coordinates
on $\sH^2$. The branch locus is the fixed point set of $\sigma$ on
$\Gamma \backslash \sH^2$.

\begin{proposition}[Hausmann \cite{Hau}]
The one-dimensional part of the fixed point set of $\sigma$ on $\Gamma
\backslash \sH^2$ is the union of the modular curves $F_w$, where $w$
ranges over $\{1,D\}$ if $D$ is odd, and over $\{1, 4, D/4, D\}$ if
$D$ is even.
\end{proposition}

Next, we analyze what happens to the K3 surfaces along these modular
curves.

\begin{proposition}
Let $N$ be a natural number, and let $C$ be any component of the
inverse image of a modular curve $F_N$, under the birational
surjective morphism $\phi_D \colon \sF_{L_D} \rightarrow \sH_D$. The K3
surface corresponding to a generic point of $C$ has N\'eron--Severi
group of rank $19$ and discriminant $2N$.
\end{proposition}

\begin{proof}
By \cite{Mc}, the generic point on any component of a modular curve
$F_N$ corresponds to an abelian surface $A$ whose ring of
endomorphisms is a quaternionic order $R$ of discriminant $N^2$. By
\cite{Har}, the N\'eron--Severi group of $A$ and the endomorphism ring
are connected by the relation $R = \sC^+\big(\tfrac{1}{2} \NS(A)\big)$,
the even part of the associated Clifford ring. From this, it is easy
to compute that $\NS(A)$ must have rank $3$ and discriminant $2N$. By
the Shioda--Inose structure, the Picard group of the associated K3
surface must have rank $19$ and discriminant $2N$.
\end{proof}

Therefore, the part of the branch locus corresponding to the curve
$F_D$ (and $F_{D/4}$ if $D$ is even) corresponds to a sublocus of K3
surfaces having Picard group of discriminant $2D$ (respectively
$D/2$). Section $4$ of \cite{EK} explains how to correctly identify
the branch locus and the arithmetic twist; we will not restate the
method or proofs here.

The discussion for the curves $F_1$ (and $F_4$ if $D$ is even) was
omitted from \cite{EK}; we address it here.

\begin{proposition}
Let $X$ be an elliptic K3 surface with a reducible fiber of type
$\II^*$ and another of type $\III^*$. Furthermore, assume that $X$ has
a section of height $D/2$, for some natural number $D$ congruent to
$0$ or $1$ modulo $4$. Then $\NS(X)$ cannot be a lattice of rank $19$
and discriminant $2$ or $8$.
\end{proposition}

\begin{proof}
First, assume that $X$ has no other reducible fibers. Then, in
particular, there cannot be a torsion section: such a section would
have height $0$, which is impossible to obtain with only two reducible
fibers of type $E_8$ and $E_7$. Therefore, if $\NS(X)$ had
discriminant $2$, the Mordell--Weil lattice would be positive definite
of rank $2$ and discriminant $1$. The Hermite constant in dimension
$2$ is $2/\sqrt{3}$, so this would imply that there is a section of
height at most $2/\sqrt{3} \approx 1.155$, which is impossible with
the fiber configuration (the best we can do is $4 - 4/3 = 8/3 >
2.66$). Similarly, if $\NS(X)$ had discriminant $8$, we would need a
section of height $4/\sqrt{3} \approx 2.309$, which is still
forbidden.  Note that we did not use the assumption of a section of
height $D/2$. For the case when $X$ could have additional reducible
fibers, we have to do some additional analysis, but a similar proof
applies.
\end{proof}

Going back to our situation, let ${\mathcal X}$ be the family of K3
surfaces lattice polarized by $L_D$, over the base $\M_{L_D}$ (the
coarse moduli space). In the third step of the method described above,
we find a different elliptic fibration on the K3 surface ${\mathcal
  X}_\nu$ (uniformly over $\nu$ on the base), with $\II^*$ and
$\III^*$ fibers. However, according to the above proposition, for
$\nu$ on one of the curves $C$ lying over $F_1$ or $F_4$, we cannot
have such an elliptic fibration on $X_\nu$. Therefore, the Weierstrass
equation for the resulting $\II^*, \III^*$ fibration must be undefined
(or not minimal, or one of these fibers must degenerate further) for
$\nu$ on such a curve $C$. This is reflected in the denominators of
the expressions for the Weierstrass coefficients and of the
Igusa--Clebsch invariants, as well as the (numerator of the) invariant
$I_{10}$. Therefore, by adding their irreducible factors to that of
the list for discriminant $2D$ (and $D/2$ when $D$ is even), we can
then continue with Step 4 of \cite{EK}.

The remaining assertions and proofs of Section $4$ of \cite{EK} go
through without any change. Therefore, we may now proceed with the
computation of the Hilbert modular surfaces $Y_{-}(n^2)$ for small
values of $n$.  Here, we describe these surfaces for $n$ up to
$11$. The cases $2 \leq n \leq 5$ have been treated in the previous
literature on the subject, although by using completely different
methods. The cases $6 \leq n \leq 11$ show that one can quickly go
beyond the current state of the art using our new techniques. For each
$n$ treated in this paper, we also describe a tautological family of
genus $2$ curves over the moduli space, whose Jacobians are
$(n,n)$-isogenous to a pair of elliptic curves, and we give the
$j$-invariants of these elliptic curves.\footnote{More specifically,
  when the moduli space $Y_{-}(n^2)$ is rational, we give the
  $j$-invariants in terms of the parameters on the moduli space. For
  the non-rational moduli spaces, we give $j_1 + j_2$ and $j_1 j_2$ in
  terms of the parameters on the (rational) Humbert surface
  $\sH_{n^2}$, and verify that $j_1 - j_2$ generates the function
  field of its double cover $Y_{-}(n^2)$.} To accomplish this, we use
the Eigenform Location Algorithm of \cite{KM1} to produce a
tautological family of {\em normalized} genus $2$ curves (i.e.~for
which the eigenforms for real multiplication are given by $dx/y$ and
$x\,dx/y$), along with a method which involves integrating an
eigenform. In principle, we can even give the two maps of degree $n$
from the genus $2$ curve to the elliptic curves; we indicate how this
is done in the case $n = 3$. There is no serious obstruction to
computing these modular surfaces for $n \geq 12$, though of course the
computations will get more challenging for large $n$.

We list the geometric type of the surfaces treated here in the
following table. For comparison with the literature, we note that
Hermann \cite{He} described the geometric classification of modular
surfaces $Y_{n, \eps}$ of discriminant $n^2$, which is the quotient of
$\sH^2$ by an appropriate subgroup of $\SL_2(\Z/n\Z)^2$. A different
proof was given by Kani and Schanz \cite{KS}, who also corrected a
typo in \cite{He}, and described the connection with the moduli space
of pairs of elliptic curves $E_1, E_2$ with an isomorphism on the
$n$-torsion $E_1[n] \cong E_2[n]$. The index $\eps \in
(\Z/n\Z)^\times$ is the factor which multiplies the determinant of the
Weil pairing under this isomorphism. See also \cite{Ca} for a nice
overview and connections to modular forms. For our particular
situation, $\eps = -1$ as the isomorphism is an anti-isometry (see
\cite{FK}), and we can read off the corresponding results from the
above papers (note that the isomorphism type only depends on the
square class of $\eps$ in $(\Z/n\Z)^\times$).

\begin{center}
\begin{tabular}{c|c}
$n$ & Geometric type \\
\hline
2 & Rational \\
3 & Rational \\
4 & Rational \\
5 & Rational \\
6 & Elliptic K3 
\end{tabular} 
\qquad \qquad 
\begin{tabular}{c|c}
$n$ & Geometric type \\
\hline
7 & Elliptic  K3 \\
8 & Honestly elliptic \\
9 & Honestly elliptic \\
10 & Honestly elliptic \\
11 & General type 
\end{tabular}
\end{center}

\section{Discriminant $4$}

\subsection{Parametrization}

We would like a family of K3 surfaces lattice polarized by $L_4$. It
is easy to see that $L_4 \cong U \oplus E_8 \oplus E_7 \oplus A_1$. We
now reverse-engineer Tate's algorithm to find elliptic K3 surfaces
with reducible fibers of types $E_8$, $E_7$ and $A_1$ at $t = \infty$,
$0$ and $1$ respectively.

A general elliptic K3 surface with section has the form 
\[
y^2 = x^3 + a(t) x^2 + b(t) x + c(t), 
\]
with $a, b, c$ being polynomials in $t$ of degree $4, 8, 12$
respectively. In order to have an $E_8$ fiber at $t = \infty$, we may
assume (after shifting $x$ suitably) that $a$ has degree at most
$2$. Similarly, the $E_7$ fiber at $t = 0$ allows us to assume that $a
= kt^2$, for some constant $k$. These normalizations involve shifting
$x$ by a linear combination of $1,t,t^3$ and $t^4$. Finally, since
there is an $A_1$ fiber at $t = 1$, we may shift $x$ by a suitable
multiple of $t^2$, to make $b(t)$ and $c(t)$ divisible by $t-1$ and
$(t-1)^2$ respectively. The final result is that the Weierstrass
equation has the form
\[
y^2 = x^3 + e t^2 x^2 + f t^3 (t-1) x + g t^5 (t-1)^2.
\]
Since we must quotient by the Weierstrass scaling of $x$ and $y$, we
see that the resulting moduli space is a weighted projective space
$\Proj(2:4:6)$ in the coordinates $e,f,g$.

\subsection{Map to $\A_2$ and equation of $\widetilde{\sL}_2$}

To compute the map to $\A_2$, we put the surface in the standard form
described in \cite{Kum1}. Here, it merely involves shifting $x$ so
that the coefficient of $x^2$ is zero, and then rescaling so that the
coefficient of $xt^3$ is $-1$.  We obtain the Weierstrass equation
\[
y^2 = x^3 - t^3\left( \frac{3f - e^2}{3} t - 1 \right)x + t^5 \left( fg t^2 -\frac{54g+9ef-2e^3}{27} t + \frac{3g + ef}{3f} \right).
\]
We may now read out the Igusa--Clebsch invariants, which are functions
of $r = f/e^2$ and $s = g/e^3$. We obtain
\begin{equation} \label{igcl4}
\begin{aligned}
I_2 &= 8(3s+r)/r,   \\
I_4 &= -4(3r-1),  \\ 
I_6 &= -4(6rs-8s+5r^2-2r)/r,   \\ 
I_{10} &= 4rs.
\end{aligned}
\end{equation}

By considering the possible ways the rank could jump and give K3
surfaces whose N\'eron--Severi lattices have the appropriate
discriminant, we obtain the following list of possible factors for the
branch locus:
\[
r,\quad s,\quad (4s - r^2),\quad (9s-3r^2+r),\quad (27s^2+36rs-s-16r^3+8r^2-r).
\]

An arithmetic verification as in \cite{EK} pins down the correct double cover.
\begin{theorem}
A birational model for the surface $\widetilde{\sL}_2$ (equivalently,
for $Y_{-}(4)$) is given by
\[
z^2 = -r(27s^2+36rs-s-16r^3+8r^2-r).
\]
It is a rational surface. The Humbert surface is birational to the
$(r,s)$-plane. In these coordinates, the Igusa--Clebsch invariants of a
point on the moduli space are given by the formulas in \eqref{igcl4}
above.
\end{theorem}
This is a conic bundle over $\Proj^1_r$. Setting $s = 0$ makes this
expression a square, so in fact this conic bundle has a section, and
is a rational surface over the base field $\Q$. To parametrize it, we
complete the square by letting $z = ms + r(4r-1)$. Solving for $s$, we
obtain
\[
s = -\frac{r(8mr+36r-2m-1)}{27r+m^2}, \quad z = \frac{r(108r^2-4m^2r-36mr-27r+m^2+m)}{27r+m^2}.
\]

\subsection{Comparison with previous formulae}

The classical formulae over an algebraically closed field are
described in \cite{SV}, for instance. We start with a genus $2$ curve
\[
y^2 = x^6 - s_1 x^4 + s_2 x^2 - 1
\]
which has two obvious maps to elliptic curves. The group $D_{12}$ acts
on the family of such Weierstrass equations as follows: its two
generators take $(x,y,s_1,s_2)$ to $(\zeta_6 x, y, s_1 \omega, s_2
\omega^2)$ or to $(1/x,iy/x^3,s_2,s_1)$, where $\zeta_6$ is a
primitive sixth root of unity, $\omega = \zeta_6^2$ and $i =
\sqrt{-1}$. The invariants of this action on the polynomial ring
generated by the parameters $s_1$ and $s_2$ are $u = s_1 s_2$ and $v =
s_1^3 + s_2^3$. Comparing Igusa--Clebsch invariants, we obtain the
relation between our coordinates $r,s$ above and $u,v$.
\begin{align*}
r &= \frac{s_1^2 s_2^2 -4 s_1^3-4 s_2^3+18 s_1 s_2-27}{4(s_1 s_2-9)^2} = \frac{u^2-4v +18u-27}{4(u-9)^2}, \\
s &= \frac{2(s_1^2 s_2^2- 4 s_1^3-4 s_2^3 + 18 s_1 s_2-27)}{(s_1 s_2-9)^3}  = \frac{2(u^2-4v+18u-27)}{(u-9)^3}
\end{align*}
with inverse
\begin{align*}
u &= (9s+8r)/s, \\
v &= 2(27s^2+36rs-32r^3+8r^2)/s^2.
\end{align*}

The branch locus, up to squares, equals $(v^2-4u^3) = (s_1^3 -
s_2^3)^2$. Recall that interchanging $s_1$ and $s_2$ changes the
equation of the genus $2$ curve by $x \to 1/x$, and so switches the
two elliptic subfields. So it agrees with the classical double
cover. Another explicit way to see this is through the $j$-invariants,
which we compute next.

\subsection{Tautological genus $2$ curve and elliptic curves}

Over the moduli space $Y_{-}(4)$ (or $\widetilde{\sL}_2$), it is
possible to write down a tautological family of genus $2$ curves with
$(2,2)$ reducible Jacobian. In terms of $r,s,z$ above, the sextic
defining the family is
\[
y^2 = x^6 + (9s+8r)x^4 -(8rz-27s^2-36rs+32r^3-8r^2)x^2 -s(8rz-27s^2-36rs+32r^3-8r^2).
\]

The $j$-invariants corresponding to the two elliptic subfields are 
\[
\pm \frac{128(4r-1)}{s^2} z - \frac{64(36rs-s-32r^3+16r^2-2r)}{s^2},
\]
which again shows that the involution switching the $j$-invariants is
the one corresponding to the double cover.

Since the Hilbert modular surface is rational with parameters $r,m$,
we may further simplify the tautological family above.

\begin{theorem}
In terms of the parameters $(r,m)$ on the Hilbert modular surface, a
tautological family of genus $2$ curve with $(2,2)$-reducible Jacobian
is given by
\[
y^2 = x^6-(2m+3)x^4-(36r-4m-3)x^2+4(2m+9)r-2m-1.
\]
\end{theorem}

\subsection{Special loci}
We now describe some special curves on the Hilbert modular surface.
\begin{enumerate}
\item The genus $0$ curve $r = 0$ is part of the branch locus; it is
  also part of the product locus, i.e.~the boundary $(\A_2 \backslash
  \M_2) \cap \sH_4$ which corresponds to products of elliptic curves
  rather than Jacobians of genus $2$ curves.
\item For the curves $(r,s,z) = (r,0,\pm r(4r-1))$, the invariant
  $I_{10} =0$. They correspond to degenerations of the genus $2$ curve
  to a rational curve.
\item The curve $s = r^2/4$ on the Humbert surface pulls back to a
  genus $0$ curve on $Y_{-}(4)$, isomorphic to $\Proj^1$. It is a
  modular curve, corresponding to Jacobians with endomorphisms by a
  (split) quaternion algebra. The minimal degree of the isogeny
  between the two elliptic curves we have associated above is
  $3$. This curve is the same as that defined in \cite[formula
    (11)]{SV}.
\item The curve $r = 1/4$ also lifts to a $\Proj^1$ on the Hilbert
  modular surface, corresponding to the modular curve for which $j_1 =
  j_2$.
\item The curve $r = -1/2, z = \pm (27 s - 2)/18$ corresponds to one
  of the $j$-invariants being $0$ (i.e.~having CM by $\Z[\omega]$).
\item The curve $27s^2+36rs-s-16r^3+8r^2-r = 0$ has genus $0$ and is
  part of the branch locus. It is parametrized by $(r,s) =
  \big((t^2-1)/12, (t-2)^2(t+1)/54 \big)$ and corresponds to the
  $j$-invariants being equal.
\item The curve $3375s^2-1440rs-152s-64r^3+48r^2-12r+1 = 0$ lifts to a
  genus $0$ curve on the Hilbert modular surface; its parametrization
  on the Humbert surface is given by $(r,s) = \big((t^2 - 1)/60$, \\
  $(t+4)^3/3375 \big)$. It corresponds to the elliptic curves being
  $2$-isogenous.
\item There are also several simple non-modular curves, for example
  the curves $s = 1/27$, $s = 2/27$, and $s = r^2/3 - r/9$. These lift
  to genus $0$ curves (in fact, isomorphic to $\Proj^1$) on
  $Y_{-}(4)$.
\end{enumerate}
\begin{remark}
Since the Hilbert modular surface is rational, in fact we have a
$2$-parameter family of genus $2$ curves with $(2,2)$-split Jacobian
given by the equation above, and we can specialize it to produce
infinitely many examples of $1$-parameter families of such genus
$2$-curves.
\end{remark}

\begin{remark}
To check the statement that elliptic curves corresponding to points on
a specific curve on $\widetilde{\sL}_2$ are (say) $2$-isogenous, one
can simply substitute in the $j$-invariants into the appropriate
classical modular polynomial. Another way to verify that such a curve
on the Hilbert modular surface is modular is to check that the Picard
number of the elliptic K3 surface jumps (for instance, due to an extra
reducible fiber, or an extra section). From now on, we shall make such
statements without further justification; the interested reader may
carry out the (easy) check.
\end{remark}

\section{Discriminant $9$}

\subsection{Parametrization}

Next, we describe degree $3$ elliptic subfields, which were treated by
Shaska in \cite{Sh}. A simple family of elliptic K3 surfaces which
will allow us to recover this moduli space is that with $E_8$ and
$A_8$ fibers. As before, we put the $E_8$ fiber at $\infty$, letting
us write the Weierstrass equation as
\[
y^2 = x^3 + a(t)x^2 + 2b(t)x + c(t)
\]
with $a,b,c$ polynomials of degrees $2,4,7$ respectively. Assuming the
$A_8$ fiber is at $t = 0$, and shifting $x$ by a suitable quadratic
polynomial, we can rewrite the Weierstrass equation as
\[
y^2 = x^3 + (a_0 + a_1 t + a_2 t^2) x^2 + 2t^3 (b_0 + b_1 t) x + t^6 (c_0 + c_1 t).
\]
This surface generically has an $A_5$ fiber at $t = 0$. To ensure an
$A_8$ fiber, we need three more orders of vanishing for the
discriminant at $t = 0$. Since we need the components of the $\I_9$
fiber to be defined over the ground field, the coefficient $a_0$ is a
nonzero square; by scaling $x$ we may assume it is $1$. Similarly, by
scaling $t$, we can arrange $c_0 = c_1$.

The three orders of vanishing succesively yield $b_0 = c_0^2$, $b_1 =
b_0 (a_1 + 1)/2$, and $a_2 = (a_1 - 1)^2/4$. Relabeling $c_0 = r$ and
$a_1 = s$, we obtain the final Weierstrass equation
\[
y^2 = x^3 + \left(\frac{(s-1)^2}{4} t^2+ s t + 1 \right)x^2 + t^3
r\Big( (s + 1)t + 2\Big)x + r^2t^6(t+1)
\]
over the rational moduli space with parameters $r,s$.

\subsection{Map to $\A_2$ and equation of $\widetilde{\sL}_3$}

To compute the map to $\A_2$, we go to an elliptic fibration with
$E_8$ and $E_7$ fibers via a $2$-neighbor step (see \cite{Kum2, EK}).

In terms of Dynkin diagrams of rational curves on the K3 surface, the
picture is as follows:

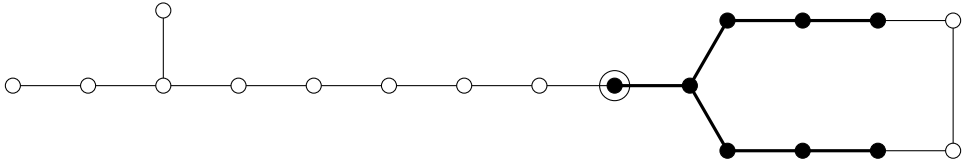
\begin{center}
\begin{tikzpicture}

\draw (-2,0)--(7,0)--(7.5,0.866)--(7.5,0.866)--(10.5,0.866)--(10.5,-0.866)--(7.5
,-0.866)--(7,0);
\draw (0,0)--(0,1);
\draw [very thick] (9.5, 0.866)--(7.5,0.866)--(7,0)--(7.5,-0.866)--(9.5,-0.866);
\draw [very thick] (6,0)--(7,0);

\draw (6,0) circle (0.2);
\fill [white] (-2,0) circle (0.1);
\fill [white] (-1,0) circle (0.1);
\fill [white] (0,0) circle (0.1);
\fill [white] (1,0) circle (0.1);
\fill [white] (2,0) circle (0.1);
\fill [white] (3,0) circle (0.1);
\fill [white] (4,0) circle (0.1);
\fill [white] (5,0) circle (0.1);
\fill [black] (6,0) circle (0.1);
\fill [white] (0,1) circle (0.1);
\fill [black] (7,0) circle (0.1);
\fill [black] (7.5,0.866) circle (0.1);
\fill [black] (7.5,-0.866) circle (0.1);
\fill [black] (8.5,0.866) circle (0.1);
\fill [black] (8.5,-0.866) circle (0.1);
\fill [black] (9.5,0.866) circle (0.1);
\fill [black] (9.5,-0.866) circle (0.1);
\fill [white] (10.5,0.866) circle (0.1);
\fill [white] (10.5,-0.866) circle (0.1);

\draw (-2,0) circle (0.1);
\draw (-1,0) circle (0.1);
\draw (0,0) circle (0.1);
\draw (1,0) circle (0.1);
\draw (2,0) circle (0.1);
\draw (3,0) circle (0.1);
\draw (4,0) circle (0.1);
\draw (5,0) circle (0.1);
\draw (6,0) circle (0.1);
\draw (0,1) circle (0.1);
\draw (7,0) circle (0.1);
\draw (7.5,0.866) circle (0.1);
\draw (7.5,-0.866) circle (0.1);
\draw (8.5,0.866) circle (0.1);
\draw (8.5,-0.866) circle (0.1);
\draw (9.5,0.866) circle (0.1);
\draw (9.5,-0.866) circle (0.1);
\draw (10.5,0.866) circle (0.1);
\draw (10.5,-0.866) circle (0.1);

\end{tikzpicture}
\end{center}

The dark vertices form a sub-diagram cutting out an $E_7$ fiber. The
corresponding elliptic parameter (unique up to fractional linear
transformations) is $w = (x + rt^3)/t^4$. Substituting $x = wt^4 -
rt^3$ in to the Weierstrass equation, and dividing the right hand side
by $t^8$, we obtain a quartic in $t$, giving rise to a genus $1$ curve
over $\Proj^1_w$. In fact, the genus $1$ fibration has a section (it
is evident from the diagram above), so we may convert to the Jacobian
using classical formulas (see \cite{AKMMMP}, for instance). After some
scaling and normalization, we arrive at a Weierstrass equation in
standard form, with the following equation
\begin{align*}
Y^2 &= X^3 + \Big( -27(s^4-4s^3+6s^2-48rs-4s+192r+1)T^4 - 5184 T^3 \Big)X+ 15552 (s^2+4s-2)  T^5 \\
    & \qquad + 54(s^6-6s^5+15s^4-72rs^3-20s^3-432rs^2+15s^2+1080rs-6s+864r^2-576r+1) T^6 \\
&  \qquad + 46656r^3    T^7 .
\end{align*}

From this we may read out the Igusa--Clebsch invariants.
\begin{equation} \label{igcl9}
\begin{aligned}
I_2 &= 8(s^2+4s-2), \\
I_4 &= 4(s^4-4s^3+6s^2-48rs-4s+192r+1), \\
I_6 &= 8(s^6+2s^5-21s^4-40rs^3+44s^3+144rs^2-41s^2+792rs+18s-288r^2-320r-3), \\
I_{10} &= 2^{14} r^3 .
\end{aligned}
\end{equation}

A similar analysis as for discriminant $4$ computes the double cover
of the Humbert surface giving the Hilbert modular surface.
\begin{theorem}
A birational model for the surface $\widetilde{\sL}_3$ (equivalently,
for $Y_{-}(9)$) is given by
\[
z^2 = 11664r^2-8(54s^3+27s^2-72s+23)r+ (s-1)^4(2s-1)^2.
\]
It is a rational surface. The Humbert surface is birational to the
$(r,s)$-plane. In these coordinates, the Igusa--Clebsch invariants of a
point on the moduli space are given by the formulas in \eqref{igcl9}
above.
\end{theorem}

The equation above describes a conic bundle over $\Proj^1_s$. Setting
$r = 0$ makes the right side a square, so the conic bundle has a
section. To parametrize the surface, we set $z = 4mr + (s-1)^2(2s-1)$
and solve for $r$, obtaining
\begin{align*}
r &= -\frac{(s-1)^2(2s-1)m + (54s^3+27s^2-72s+23)}{2(m-27)(m+27)}, \\ 
z &= -\frac{(s-1)^2(2s-1)(m^2 + 729) + 2(54s^3+27s^2-72s+23)m}{(m-27)(m+27)}.
\end{align*}

\subsection{Tautological curve}

A tautological curve over the Hilbert modular surface (which is
rational in the parameters $m,s$ above) is given by
\[
y^2 = \left( x^3 + 3(3s-1)x^2  -\frac{2(m-27)(9s-5)^2}{(m+27)} \right)  \left( x^3 -\frac{3(m-27)(9s-5)^2}{4(m+27)}  x + \frac{(m-27)(9s-5)^3}{4(m+27)} \right).
\]

\begin{remark} This presentation of the curve has the feature that the
eigen-differentials for the action of real multiplication by $\sO_9$
(alternatively, the pullbacks of the canonical differentials from the
two elliptic curves) are $dx/y$ and $x \, dx/y$. We say that the
Weierstrass equation is normalized.\footnote{There is still an extra
  degree of freedom from scaling the $x$-coordinate, and we make an
  arbitrary choice, designed to make the equation look ``nice''.} The
Weierstrass equations given in this paper and in the auxiliary files
for tautological families of genus $2$ curves (for each of the
discriminants) will always be normalized. To check this property of
being normalized amounts to a calculation using the ``eigenform
location'' algorithm of \cite{KM1}. In this paper, we omit the
details.
\end{remark}

\subsection{Elliptic curves}

Next, we compute the $j$-invariants associated of the two associated
elliptic curves. 

\subsubsection{Calculation of $j$-invariants}
To calculate the $j$-invariants in this and other sections, we used
the following computational technique. For many specializations $(r,s)
\in \Q^2$, we compute the genus $2$ curve from its Igusa--Clebsch
invariants. We then decompose its Jacobian and find out its simple
factors, using analytic techniques (using the function
\texttt{AnalyticJacobian} in \texttt{Magma}; see \cite{Wa2}). More
precisely, let $A$ be the analytic Jacobian. If the coordinates $r,s$
are chosen generically, its endomorphism ring will be $\sO_9$, the
quadratic ring of discriminant $9$. We solve the equation $\eta^2 =
3\eta$ in $\End(A)$. Then we can recover the period matrices of $E_1 =
A/\eta A$ and $E_2 = A/(\eta - 3)A$, and from there compute the two
$j$-invariants. In general, these are quadratic over the base field,
but their sum and product are rational. We use rational reconstruction
to recover these numbers $(j_1 + j_2)(r,s)$ and $(j_1 j_2)(r,s)$ from
floating-point approximations. Finally, we apply the above procedure
to enough specializations to enable us to recover $j_1 + j_2$ and $j_1
j_2$ as rational functions, by interpolation/linear algebra (using
some upper bounds on the degree as in \cite{Wa1}) .

We obtain the following expressions:
\begin{align*}
j_1 + j_2 &= (2 s^9-17 s^8+64 s^7-324 r s^6-140 s^6+1350 r s^5+196 s^5-2097 r s^4 \\
 & \qquad -182 s^4  +17496 r^2 s^3+1368 r s^3+112 s^3-23328 r^2 s^2-162 r s^2-44 s^2 \\
 & \qquad       +9720 r^2 s-198 r s+10 s-314928 r^3-432 r^2+63 r-1)/r^2 , \\
j_1 j_2 &= (s^4-4 s^3+6 s^2+432 r s-4 s-288 r+1)^3/r^3.
\end{align*}
The discriminant of $T^2 - T(j_1 + j_2) + j_1 j_2$ is, up to squares,
the branch locus of the double cover defining the Hilbert modular
surface. Solving in terms of $m$ and $s$, we have (up to interchange)
\begin{align*}
j_1 &= \frac{4 (m+27) (m s^2+27 s^2-2 m s+18 s+m-21)^3}{(m-27) (2 m s^3+54 s^3-5 m s^2+27 s^2+4 m s-72 s-m+23)}, \\
j_2 &= \frac{-2 (m-27) (m s^2-459 s^2-2 m s+486 s+m-123)^3}{(m+27) (2 m s^3+54 s^3-5 m s^2+27 s^2+4 m s-72 s-m+23)^2}.
\end{align*}

The above process may seem non-rigorous, but in fact can be made
completely rigorous, for instance, by demonstrating the real
multiplication by $\sO_9$ through an explicit correspondence (as
described in \cite{Wa1} or in more detail in \cite{KM2}). For reasons
of space, we do not describe this method explicitly here.

\subsubsection{Calculation of the morphism}

However, we will describe another method, suggested to us by Noam
Elkies, which makes use of the fact that the differentials on the
genus $2$ curve are normalized (i.e.~are eigenforms for real
multiplication).

We start with the (arbitrarily chosen) point $s = 5, m = 13$ on
$\widetilde{\sL}_3$. Rescaling the $x$-coordinate so the Weierstrass
equation has integer coefficients, the genus $2$ curve is given by
\[
y^2 = (x^3+420x-5600)(x^3+42x^2+1120).
\]
For convenience, let us instead start from the curve $C$
\[
y^2 = (-5600x^3+420x^2+1)(1120x^3+42x+1)
\]
(obtained by the change of variables $x \to 1/x, y \to y/x^3$) which
has the rational point $(0,1)$. The form $x \, dx/y$ is a normalized
differential on $C$ (since $dx/y \mapsto (-1/x^2)\, dx/(y/x^3) = -x \, dx/y$
under the above transformation). Assume that the associated elliptic
curve $E$ is given by
\[
y_1^2 = F x_1^3 + G x_1^2 + H x_1 + 1
\]
where $F$, $G$ and $H$ are (for the moment) undetermined
parameters. Also, let us assume without loss of generality that the
morphism $\phi \colon C \to E$ which we want to determine takes $(0,1) \in
C$ to $(0,1) \in E$. Note that the above model for $E$ can easily be
converted to standard Weierstrass form.

Our normalization assumption says that $\phi^*$ maps the canonical
differential $dz_1 = dx_1/y_1$ on $E$ to $dz = x \, dx/y$ on
$C$. Expanding in the formal ring about $(0,1)$ on these curves, we
compute as follows:
\[
\frac{x \, dx}{y} = \frac{x \, dx}{\big((-5600x^3+420x^2+1)(1120x^3+42x+1)\big)^{1/2}} = (x - 21x^2 + 903x^3/2 + \dotsb ) \, dx.
\]
Therefore 
\[
z = \int x \, dx/y = x^2/2 - 7 x^3 + \dotsb .
\]
On the other hand, we have 
\[
\frac{ dx_1}{y_1} = \frac{ dx_1}{( F x_1^3 + G x_1^2 + H x_1 + 1)^{1/2}} = \left(1 - \frac{H}{2} \, x_1 + \bigg(\frac{-G}{2} + \frac{3H^2}{8} \bigg) \, x_1^2 + \dotsb \right) dx_1 ,
\]
so 
\[
z_1 = \int dx_1/y_1 = x_1 - (H/4) x_1^2 + \dotsb.
\]
Now, since $dz = \phi^*(dz_1)$, we have the equation $z =
\phi^*(z_1)$, which we can write using the above expressions as
\[
x^2/2 - 7x^3 + \dotsb = \phi^*\big( x_1 - (H/4) x_1^2 + \dotsb \big) = x_E - (H/4) x_E^2 + \dotsb,
\]
where $x_E = \phi^*(x_1)$. We can invert the formal series on the
right to write
\[
x_E = x^2/2 - 7x^3 + (H/16 + 903/8)x^4 + \dotsb .
\]
Now, we know that $x_E = \phi^*(x_1)$ must be a rational function of
degree $3$ in $x$ (for a degree $n$ map $C \to E$, it will have degree
$n$).  Since its formal expansion starts $x^2/2 + \dotsb$, it must
equal
\[
\frac{x^2(1+kx)}{2(1 + mx + nx^2 + px^3)}
\]
for some constants $k,m,n,p$ which are also undetermined.

We expand the above rational expression in a power series about $x =
0$ and match with the previous formal expression for $x_E$. The first
few coefficients give us linear equations which can be solved for
$m,n,p$ and $k$ in that order, and then we can solve the next few
polynomial equations for $F,G,H$. Renormalizing the equations of $C$
and $E$, we find that the elliptic curve $E$ is given by the
Weierstrass equation
\[
y_1^2 = x_1^3 + 4900 x_1^2 + 7031500 x_1 + 2401000000,
\]
and the morphism $C \to E$ by 
\[
x_1 =  -\frac{882000(x-14)}{x^3+420x-5600}, \qquad 
y_1 =\frac{ 49000 y (x^3-21x^2-140)}{(x^3+420x-5600)^2}.
\]
(Note that the $j$-invariant of $E$ is $-2^5 \cdot 7 \cdot 17^3$,
which matches the expression given for $j_1$ at $s = 5, m = 13$.)

\subsubsection{Further remarks}

The above calculation produced the morphism $C \to E_1$ to one of the
associated elliptic curves, for a single point on the moduli space. To
produce it over the whole base, one may work with the tautological
genus $2$ curve, and with coefficients in the function field of
$\widetilde{\sL}_3$. Alternatively, we may sample many such points and
interpolate to find the coefficients of the rational map defining the
morphism, and also the Weierstrass coefficients of the elliptic
curve. Once the map is produced, it is trivial to verify it (by
substituting in to the equation of the elliptic curve). To find the
other map $C \to E_2$, we may use the other normalized differential on
the genus $2$ curve and apply the above process to it. More simply, we
may rewrite the equation of $E_1$ and the map $C \to E_1$ in terms of
the coordinates $r,s$ on the Humbert surface and the square root $z$
defining the double cover, and conjugate $z$ to $-z$ everywhere.

We omit the details of these calculations for the higher discriminants
treated in this paper, as the expressions become much more complicated
with increasing discriminant. We may, in fact, use the knowledge of
the $j$-invariants obtained by the first (sampling and interpolation)
method to reduce our work in this process, since they give us an extra
equation satisfied by the Weierstrass coefficients $F,G,H$.

Finally, the above process hinges on finding normalized Weierstrass
forms for the genus $2$ curves in the first place. For a single point
$(r_0,s_0,z_0)$ on the moduli space (with $r_0, s_0 \in \Q$ say),
i.e.~for a single genus $2$ curve, we do this by applying the
eigenform location algorithm of \cite{KM2}, which indicates the
fractional linear transformation needed to transform the curve to
normalized form. At that point, the Weierstrass coefficients are
floating point numbers, but we may apply a version of rational
reconstruction to recognize them as elements of $\Q(z_0)$. Finally,
with enough sampling, we may interpolate to write a normalized
tautological curve over the moduli space.

\subsection{Comparison with previous results}

Comparing our absolute invariants $I_4/I_2^2, \, I_2 I_4/I_6,\, I_4
I_6/I_{10}$ to those of Shaska, we obtain the relation between the
coordinates here and those in \cite{Sh}. It is given by
\[
r = 4096 r_2^3/r_1^4, \qquad s = -48r_2/r_1,
\]
with inverse
\[
r_1 = -s^3/(27r), \qquad r_2 = s^4/(1296 r). 
\]
Using this change of coordinates, we see that the double cover is the
same as the $(u,v)$ double cover in \cite{Sh}\footnote{There is a typo
  in equation (14) of \cite{Sh}: the numbers $27$ and $1296$ should be
  in the denominator rather than the numerator.}.

Let us check the compatibility of the $j$-invariants we found above
with the values in \cite{Sh}. In the setup of that paper, the genus
$2$ curve is given by
\[
y^2 = (x^3 + ax^2 + bx + 1)(4x^3 + b^2x^2 + 2bx + 1),
\]
where $a,b$ are related to the parameters $u,v$ on the Hilbert modular surface by 
\[
u = ab, \qquad v = b^3.
\]
Note that the parameters $r_1, r_2$ on the Humbert surface are 
\[
r_1 = \frac{v (v-2 u-9)^3}{27 (4 v^2-u^2 v-18 u v+27 v+4 u^3)}, \qquad
r_2 = -\frac{v (v-2 u-9)^4}{1296 (v-27) (4 v^2-u^2 v-18 u v+27 v+4 u^3)}.
\]
Define\footnote{There is a typo on pg. 266 of \cite{Sh}: in equation
  (6), the first factor in $\Delta$ should be the quantity $R$ as we
  have defined it, rather than its square.} $R = 4a^3 + 27 - 18ab -
a^2b^2 + 4b^3$ , $F(x) = x^3 + ax^2 + bx + 1$, and consider the
rational functions $v = y(x^3 - bx -2)/F(x)^2$ and $u = x^2/(x^3 +
ax^2 + bx + 1)$. Then $C$ maps to the genus $1$ curve
\[
v^2 = u^3 + \frac{2(ab^2 - 6a^2 + 9b)}{R}\,  u^2  + \frac{12a - b^2}{R} \, u - \frac{4}{R} .
\]
Computing the $j$-invariant of this elliptic curve, and using the
transformation formulas
\[
u = (ms+27s-m+9)/4, \qquad v = -(m-27)/2,
\]
we see that it agrees with $j_2$.

\subsection{Special loci}
We now list several curves of interest on the Hilbert modular surface.
\begin{enumerate}
\item The rational curves $r = 0, z = \pm (s-1)^2 (2s -1)$ correspond
  to $I_{10} = 0$.
\item The curve $r = (s-1)^2/4$ on the Humbert surface lifts to a
  rational curve, whose points correspond to Jacobians with
  endomorphisms by a split quaternion algebra. The associated elliptic
  curves are related by a $5$-isogeny.
\item The curves $s = 5/9, z = 4(19683r - 4)/729$ are non-modular, but
  correspond to the ``degenerate'' case considered by Shaska in
  \cite[equation (12)]{Sh}.
\item The modular genus $0$ curve
\[
4s^6-20s^5+41s^4-432rs^3-44s^3-216rs^2+26s^2+576rs-8s+11664r^2-184r+1 = 0
\]
is part of the branch locus. It is parametrized by $(r,s) =
\big( (t+1)^2(t+4)^4/(729t^6),(11t^2+24t+16)/(6t^2) \big)$ and
corresponds to the two elliptic curves being isomorphic.
\item The curves $s = 1$, $s = 1/3$ lift to non-modular rational
  curves on $Y_{-}(9)$.
\end{enumerate}

\section{Discriminant $16$}

\subsection{Parametrization}

We start with a family of elliptic K3 surfaces with bad fibers of
types $E_7$ and $D_8$, and a section of height $2 = 4 - 2$. A
Weierstrass equation for such a K3 surface has the form
\[
y^2 = x^2 + t a(t) x^2 + t^4 b(t) x+ t^7 c
\]
with $a,b$ linear in $t$ and $c$ constant. Since $a(0) \neq 0$ to have
a fiber of type no worse than $D_8$ at $0$, we may write
\[
a = a_0(1 + a_1 t), \qquad b = a_0 (b_0 + b_1t), \qquad c = c_0 a_0.
\]
We have a degree of freedom (scaling $t$), so we scale to have $a_1
=1$. To have the leaves of the $D_8$ fiber defined over the ground
field, we set $b_0 = h + k$ and $c_0 = hk$. Next, we impose a section
of height $2$: its $x$-coordinate must have the form $t^3(-h + g^2t)$
for some $g,h$. Substituting in to the Weierstrass equation and
completing the square, we get a system of equations, which can be
solved for $b_1$ and $k$. We finally use the remaining degree of
freedom (Weierstrass scaling of $x$ and $y$) to set $g = 1$, since the
variable $g$ has weight $1$. Thus, we get a rational moduli space in
the parameters $a_0$ and $h$. The Weierstrass equation suggests
the convenient change of variables $h = 2s, a_0 = 2r + h^2/4$.

The final Weierstrass equation is 
\[
y^2 = x^3 + (s^2 + 2r)t(t+1)x^2 - t^4\big( (s^2 - 2rs + 2r)t - (2s^3 + 4rs + r^2) \big) x  + 2r^2 s t^7.
\]
Note that the section has $x$-coordinate $-2s t^3 + t^4$.

\subsection{Map to $\A_2$ and equation of $\widetilde{\sL}_4$}

We go to an $E_8 E_7$ fibration by a $2$-neighbor step, as illustrated
below.

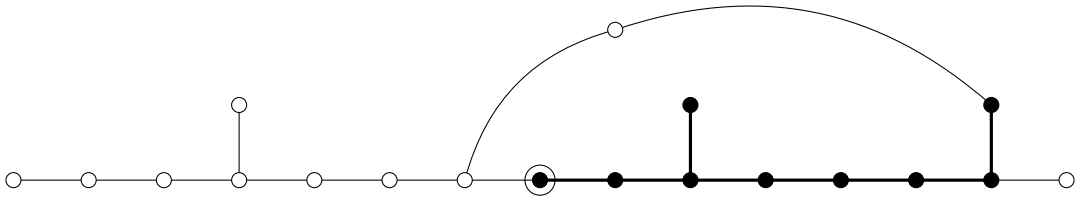
\begin{center}
\begin{tikzpicture}

\draw (-2,0)--(12,0);
\draw (1,0)--(1,1);
\draw (7,0)--(7,1);
\draw (11,0)--(11,1);
\draw [very thick] (5,0)--(11,0)--(11,1);
\draw [very thick] (7,0)--(7,1);
\draw [bend right] (6,2) to (4,0);
\draw [bend left] (6,2) to (11,1);

\draw (5,0) circle (0.2);
\fill [white] (-2,0) circle (0.1);
\fill [white] (-1,0) circle (0.1);
\fill [white] (0,0) circle (0.1);
\fill [white] (1,0) circle (0.1);
\fill [white] (2,0) circle (0.1);
\fill [white] (3,0) circle (0.1);
\fill [white] (4,0) circle (0.1);
\fill [black] (5,0) circle (0.1);
\fill [black] (6,0) circle (0.1);
\fill [white] (1,1) circle (0.1);
\fill [black] (7,0) circle (0.1);
\fill [black] (8,0) circle (0.1);
\fill [black] (9,0) circle (0.1);
\fill [black] (10,0) circle (0.1);
\fill [black] (11,0) circle (0.1);
\fill [white] (12,0) circle (0.1);
\fill [black] (7,1) circle (0.1);
\fill [black] (11,1) circle (0.1);
\fill [white] (6,2) circle (0.1);

\draw (-2,0) circle (0.1);
\draw (-1,0) circle (0.1);
\draw (0,0) circle (0.1);
\draw (1,0) circle (0.1);
\draw (2,0) circle (0.1);
\draw (3,0) circle (0.1);
\draw (4,0) circle (0.1);
\draw (5,0) circle (0.1);
\draw (6,0) circle (0.1);
\draw (1,1) circle (0.1);
\draw (7,0) circle (0.1);
\draw (8,0) circle (0.1);
\draw (9,0) circle (0.1);
\draw (10,0) circle (0.1);
\draw (11,0) circle (0.1);
\draw (12,0) circle (0.1);
\draw (7,1) circle (0.1);
\draw (11,1) circle (0.1);
\draw (6,2) circle (0.1);

\end{tikzpicture}
\end{center}

An elliptic parameter is $w = \big(x + r^2 t^3/(s^2 + 2r)\big)/t^4$.
As in the previous section, we transform to the new elliptic
fibration, which has the desired $E_8$ and $E_7$ fibers, and proceed
to read out the Igusa--Clebsch invariants. They are given as follows:
\begin{equation} \label{igcl16}
\begin{aligned}
I_2 &= 2(3s^6-12rs^5+13r^2s^4+18rs^4-2r^3s^3-48r^2s^3+28r^3s^2 \\
    & \qquad    +36r^2s^2+2r^4s-48r^3s+4r^4+24r^3)/\big(r^2(s^2-2rs+2r)\big), \\
I_4 &= (s^4+24s^3+4rs^2+48rs+r^2)/4, \\
I_6 &= (8s^{10}-32rs^9+192s^9+34r^2s^8-688rs^8-4r^3s^7+624r^2s^7+1536rs^7 \\
    &\qquad -16r^3s^6-4294r^2s^6-8r^4s^5+3000r^3s^5+4608r^2s^5-109r^4s^4 \\
    &\qquad -8612r^3s^4+26r^5s^3+3488r^4s^3+6144r^3s^3+124r^5s^2-5576r^4s^2 \\
    &\qquad +4r^6s+480r^5s+3072r^4s+12r^6+208r^5)/\big(16r^2(s^2-2rs+2r)\big), \\
I_{10} &= r^2(s^2+2r)(s^2-2rs+2r)/256.
\end{aligned}
\end{equation}

We then compute the double cover defining the Hilbert modular surface.
\begin{theorem}
A birational model for the surface $\widetilde{\sL}_4$ (equivalently,
for $Y_{-}(16)$) is given by
\[
z^2 = 2s(s^2-2rs+2r)(2s^5-4rs^4-27s^4+76rs^3-32r^2s^2-108rs^2+144r^2s-64r^3 -108r^2).
\]
It is a rational surface. The Humbert surface is birational to the
$(r,s)$-plane. In these coordinates, the Igusa--Clebsch invariants of a
point on the moduli space are given by the formulas in \eqref{igcl16}
above.
\end{theorem}

The right hand side is quartic in $r$, and we easily calculate that
the surface is a rational elliptic surface with section. To find a
rational parametrization, we make some elementary transformations to
reduce the degree of the equation in $r$ and $s$. The net effect of
these is to set
\[
r = \frac{-1}{e(e+2)^2f(2ef-1)}, \qquad 
s = \frac{1}{e(e+2)f}, \qquad 
z = \frac{w}{e^3(e+2)^5f^4(2ef-1)^2}
\]
The double cover is transformed to 
\[
w^2 = (108e^3+160e^2-144e)f^2+(-54e^2-88e+8)f+4,
\]
which is a conic bundle over the $e$-line with an obvious section
$(f,w) = (0,2)$. Parametrizing with a new parameter $d$, we obtain
finally
\begin{align*}
r &= -(108e^3+160e^2-144e-d^2)^2/\big(2e(e+2)^2(16e^2+8de+128e+d^2)(27e^2+44e+2d-4)\big), \\
s &= (108e^3+160e^2-144e-d^2)/\big(2e(e+2)(27e^2+44e+2d-4)\big).
\end{align*}
where we have omitted $z$ for simplicity.

\subsection{Tautological curve}

Once again, it is possible to write down the tautological family of
genus $2$ curves over the Hilbert modular surface. However, the
expression is quite long, and so we will omit it in the main body of
the paper. However, it can be obtained from the auxiliary computer
files. Similarly, for the higher discriminants treated in this paper,
we will omit the expression for the tautological family of genus $2$
curves (these are also in the computer files).

\subsection{Elliptic curves}

As before, it is possible to compute the equations of the two elliptic
curves attached to the $(4,4)$-decomposable Jacobian, and the
morphisms from the tautological genus $2$ curve to the elliptic
curves. We omit this calculation, and simply display the elementary
symmetric functions in the $j$-invariants of the two elliptic curves.

\begin{align*}
j_1 + j_2 &= 128(4s^{11}-8rs^{10}-63s^{10}+200rs^9+184s^9-128r^2s^8-1144rs^8 \\
 & \qquad \quad +1792r^2s^7+1392rs^7-768r^3s^6-5372r^2s^6+5888r^3s^5 \\
 & \qquad \quad +3936r^2s^5-2048r^4s^4-9840r^3s^4+7168r^4s^3+4928r^3s^3 \\
 & \qquad \quad -2048r^5s^2-6464r^4s^2+2048r^5s+2304r^4s-256r^5)/(s^2+2r)^4, \\
j_1 j_2 &= 4096(s^4+24s^3-56rs^2+48rs+16r^2)^3/(s^2+2r)^4. 
\end{align*}

Over the Hilbert modular surface, which is rational in parameters
$d,e$, we can solve the quadratic equation to get the two
$j$-invariants.
\begin{align*}
j_1 &= 16(1080e^3+108de^2-976e^2-32de-96e-d^2)^3/\big(e^2(16e^2+8de+128e+d^2)(27e^2+44e+2d-4)^3\big), \\
j_2 &= 16(24e^3-12de^2-16e^2-16de-96e-d^2)^3/\big(e^2(e+2)^4(16e^2+8de+128e+d^2)(27e^2+44e+2d-4)\big).
\end{align*}

\subsection{Comparison with previous work}

We compare our formulas with those of Bolza \cite[pp. 477, 480]{Kr},
since those of Bruin and Doerksen are equivalent \cite[Appendix
  A]{BD}. Bolza gives a family of genus $2$ curves with
$(4,4)$-reducible Jacobians over a weighted projective space with
$\lambda,\mu,\nu$ of weights $1,2,3$ respectively. By comparing the
Igusa--Clebsch invariants, we obtain the following relation between
those coordinates and ours:
\[
\frac{\mu}{\lambda^2} = \frac{(4e+d+16)(12e+3d-16)}{3(16e^2+8de+128e+d^2)} , \qquad 
\frac{\nu}{\lambda^3} = \frac{(4e+d+16)^2}{(16e^2+8de+128e+d^2)}.
\]
with inverse
\[
d = \frac{8(7\nu^2+18\lambda\mu\nu-16\lambda^3\nu-9\lambda^2\mu^2)}{9(\nu-\lambda\mu)^2}, \qquad 
e = \frac{-2(\nu^2+6\lambda\mu\nu-16\lambda^3\nu+9\lambda^2\mu^2)}{9(\nu-\lambda\mu)^2}.
\]

\subsection{Special loci}
\begin{enumerate}
\item The curves $r=0, z = 4s^2-54s$ and $r = s^2/(2s-2), z = 0$
  belong to the product locus. The latter is also part of the branch
  locus for the double cover $\widetilde{\sL}_4 \to \sL_4$.
\item The curve $2s^3+4rs-r^2 = 0$ on $\sL$ has genus $0$, and is
  parametrized by $(r,s) = \big(t^2(t-4)/2, t(t-4)/2 \big)$. Its lift
  is also a rational curve, parametrized by letting $t =
  -(v^2+171)/\big(4(v-5)\big)$. It is a modular curve, corresponding
  to the elliptic curves being $7$-isogenous.
\item The curve $s = 0$ is part of the branch locus.
\item The curve $s^5-rs^4-s^4+2rs^3-4rs^2-4r^2 = 0$ on the Humbert
  surface has genus $0$, and is parametrized by $(r,s) = \Big(
  (t^2+2t-2)^2/\big(t^4(t-1)\big),
  (t^2+2t-2)/\big(t(t-1)\big)\Big)$. Its lift is also rational,
  parametrized by letting $t = u(1-2u)/(u^2+2)$. This modular curve
  corresponds to the elliptic curves being $3$-isogenous.
\item The curve
  $2s^5-4rs^4-27s^4+76rs^3-32r^2s^2-108rs^2+144r^2s-64r^3-108r^2 =0$
  is part of the branch locus. It has genus $0$, and is parametrized by 
\[
(r,s) = \left( \frac{27(t-1)(3t^2+10t-1)^2}{8(3t-1)^4} , \frac{9(t+1)(3t^2+10t-1) }{(3t-1)^2}  \right).
\]
\item There are also several simple non-modular curves: for instance,
  the specializations $s = 9$ and $s = 63/8$ gives genus $0$ curves
  with rational points.
\end{enumerate}

\section{Discriminant $25$}

\subsection{Parametrization}

Next, we go on to discriminant $25$, which was studied in
\cite{MSV}. We start with a family of elliptic K3 surfaces with bad
fibers of type $D_9$ and $A_6$, and a section of height $25/28 = 4 -
9/4 - 6/7$. A Weierstrass equation for such a K3 surface has the form
\[
y^2 = x^3 + (a_0 + a_1 t + a_2 t^2 + a_3 t^3) x^2 + 2t(b_0 + b_1 t + b_2 t^2) x + (c_0 + c_1t)^2 t^2.
\]
This has a $D_9$ fiber at $t = \infty$ and an $A_1$ fiber at $t =
0$. We have shifted $x$ so that the generator of the Mordell--Weil
group has $x$-coordinate $0$. We need five more orders of vanishing
for the discriminant at $t = 0$. Since the components of the $A_6$
fiber are defined over the ground field, $a_0$ is a square; we may
normalize it to be $1$. This takes care of the scaling in $x$. We then
proceed to successively set the coefficients of $t^2$ through $t^5$ in
the expression for the discriminant to be zero, which gives us the
conditions
\begin{align*}
c_0 &= b_0, \\
c_1 &= (2b_1+b_0^2-a_1b_0)/2, \\
b_2 &= -\big((8b_0-4a_1)b_1+8b_0^3-10a_1b_0^2+(3a_1^2-4a_2)b_0\big)/8, \\
a_3 &= (2b_1+3b_0^2-2a_1b_0)(4b_1+10b_0^2-8a_1b_0-4a_2+a_1^2)/(8b_0).
\end{align*}
To simplify the succeeding expressions, we define a new variable $k$
by $b_1 = k + a_1 b_0/2$. The condition for the final order of
vanishing gives us a quadratic equation in $a_2$, whose discriminant
is a square times $4k^2-5b_0^4+2a_1b_0^3$. Therefore, setting
\[
a_1 = (h^2-4k^2+5b_0^4)/(2b_0^3)
\]
makes the expression a square, and we solve for $a_2$, obtaining
\begin{align*}
a_2 &= (16 k^4-32 b_0^2 k^3-8 h^2 k^2+16 b_0^2 h k^2-24 b_0^4 k^2+8 b_0^2 h^2 k \\
  & \qquad +16 b_0^4 h k+8 b_0^6 k+h^4-4 b_0^2 h^3-2 b_0^4 h^2+4 b_0^6 h+5 b_0^8)/(16 b_0^6).
\end{align*}
We also scale $b_0 = 1$ to quotient by the scaling in $t$, and let
\[
h = s-r, \qquad k = (r+s-1)/2
\]
to simplify the expressions; this change of variable is suggested by
the final formulas below. We obtain the final Weierstrass equation
\begin{align*}
y^2 &= x^3 + \big(r^2(2r-1)s^2t^3+(4r^2s^2-4rs^2+s^2-12r^2s+10rs+r^2)t^2/4-(2rs-s-r-2)t + 1\big) x^2 \\
&\qquad -  t\big( (2rs^2-s^2+6r^2s-6rs-r^2)t^2/2 + (2rs-2s-2r-1)t - 2\big)x + t^2(st+rt+2)^2/4.
\end{align*}
As noted above, the section of height $25/28$ has $x$-coordinate $0$.

\subsection{Map to $\A_2$ and equation of $\widetilde{\sL}_5$}

We then proceed to an elliptic fibration with $E_8$ and $A_7$ fibers,
via a $2$-neighbor step. The elliptic parameter is simply $x$.

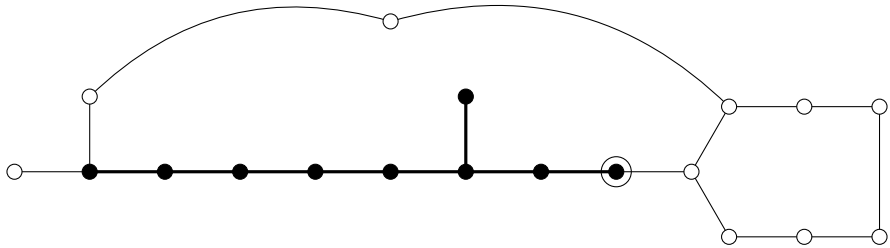
\begin{center}
\begin{tikzpicture}

\draw (-2,0)--(7,0);
\draw (-1,0)--(-1,1);
\draw (4,0)--(4,1);
\draw (7,0)--(7.5,0.866)--(9.5,0.866)--(9.5,-0.866)--(7.5,-0.866)--(7,0);

\draw (3,2) to [bend right] (-1,1);
\draw (3,2) to [bend left] (7.5,0.866);
\draw [very thick] (-1,0)--(6,0);
\draw [very thick] (4,0)--(4,1);

\draw (6,0) circle (0.2);
\fill [white] (-2,0) circle (0.1);
\fill [black] (-1,0) circle (0.1);
\fill [black] (0,0) circle (0.1);
\fill [black] (1,0) circle (0.1);
\fill [black] (2,0) circle (0.1);
\fill [black] (3,0) circle (0.1);
\fill [black] (4,0) circle (0.1);
\fill [black] (5,0) circle (0.1);
\fill [black] (6,0) circle (0.1);
\fill [white] (7,0) circle (0.1);
\fill [white] (7.5,0.866) circle (0.1);
\fill [white] (8.5,0.866) circle (0.1);
\fill [white] (9.5,0.866) circle (0.1);
\fill [white] (7.5,-0.866) circle (0.1);
\fill [white] (8.5,-0.866) circle (0.1);
\fill [white] (9.5,-0.866) circle (0.1);
\fill [white] (-1,1) circle (0.1);
\fill [black] (4,1) circle (0.1);
\fill [white] (3,2) circle (0.1);

\draw (-2,0) circle (0.1);
\draw (-1,0) circle (0.1);
\draw (0,0) circle (0.1);
\draw (1,0) circle (0.1);
\draw (2,0) circle (0.1);
\draw (3,0) circle (0.1);
\draw (4,0) circle (0.1);
\draw (5,0) circle (0.1);
\draw (6,0) circle (0.1);
\draw (7,0) circle (0.1);
\draw (7.5,0.866) circle (0.1);
\draw (8.5,0.866) circle (0.1);
\draw (9.5,0.866) circle (0.1);
\draw (7.5,-0.866) circle (0.1);
\draw (8.5,-0.866) circle (0.1);
\draw (9.5,-0.866) circle (0.1);
\draw (-1,1) circle (0.1);
\draw (4,1) circle (0.1);
\draw (3,2) circle (0.1);

\end{tikzpicture}
\end{center}

After some easy Weierstrass transformations and a fractional linear
transformation of the elliptic parameter, the Weierstrass equation for
the new fibration is
\begin{align*}
y^2 &= x^3 + \big( 12 r^3 s t^3 + (4 r^2 s^2-4 r s^2+s^2+12 r^2 s-14 r s+r^2) t^2 -4(2 r s+s+r-1) t + 4 \big) x^2/16  \\
  &\qquad  + t^4 \big( 6 r^6 s^2 t^2 + r^3 s (4 r^2 s^2-4 r s^2+s^2+12 r^2 s-14 r s+r^2) t  \\
& \qquad \qquad -2r s (4 r^2 s^2-4 r s^2+s^2+2 r^3 s+r^2 s+r^3-r^2) \big) x/32  \\
   & \qquad + t^8 \big( 4r^9 s^3  t + r^6 s^2 (4 r^2 s^2-4 r s^2+s^2+12 r^2 s-14 r s+r^2) \big)/256.
\end{align*}

There is a section $P$ of height $25/8 = 4 - 7/8$.  Finally, we take a
$2$-neighbor step to an elliptic fibration with $E_8$ and $E_7$
fibers.

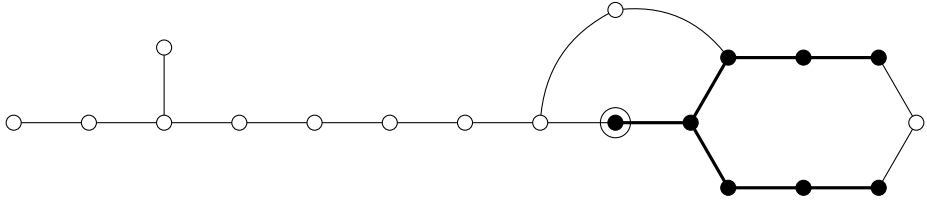
\begin{center}
\begin{tikzpicture}

\draw (-2,0)--(7,0)--(7.5,0.866)--(7.5,0.866)--(9.5,0.866)--(10,0)--(9.5,-0.866)--(7.5,-0.866)--(7,0);
\draw (0,0)--(0,1);
\draw [very thick] (9.5, 0.866)--(7.5,0.866)--(7,0)--(7.5,-0.866)--(9.5,-0.866);
\draw [very thick] (6,0)--(7,0);

\draw (6,0) circle (0.2);
\draw (6,1.5) to [bend right] (5,0);
\draw (6,1.5) to [bend left] (7.5,0.866);

\fill [white] (-2,0) circle (0.1);
\fill [white] (-1,0) circle (0.1);
\fill [white] (0,0) circle (0.1);
\fill [white] (1,0) circle (0.1);
\fill [white] (2,0) circle (0.1);
\fill [white] (3,0) circle (0.1);
\fill [white] (4,0) circle (0.1);
\fill [white] (5,0) circle (0.1);
\fill [black] (6,0) circle (0.1);
\fill [white] (0,1) circle (0.1);
\fill [black] (7,0) circle (0.1);
\fill [black] (7.5,0.866) circle (0.1);
\fill [black] (7.5,-0.866) circle (0.1);
\fill [black] (8.5,0.866) circle (0.1);
\fill [black] (8.5,-0.866) circle (0.1);
\fill [black] (9.5,0.866) circle (0.1);
\fill [black] (9.5,-0.866) circle (0.1);
\fill [white] (10,0) circle (0.1);
\fill [white] (6,1.5) circle (0.1);

\draw (-2,0) circle (0.1);
\draw (-1,0) circle (0.1);
\draw (0,0) circle (0.1);
\draw (1,0) circle (0.1);
\draw (2,0) circle (0.1);
\draw (3,0) circle (0.1);
\draw (4,0) circle (0.1);
\draw (5,0) circle (0.1);
\draw (6,0) circle (0.1);
\draw (0,1) circle (0.1);
\draw (7,0) circle (0.1);
\draw (7.5,0.866) circle (0.1);
\draw (7.5,-0.866) circle (0.1);
\draw (8.5,0.866) circle (0.1);
\draw (8.5,-0.866) circle (0.1);
\draw (9.5,0.866) circle (0.1);
\draw (9.5,-0.866) circle (0.1);
\draw (10,0) circle (0.1);
\draw (6,1.5) circle (0.1);

\end{tikzpicture}
\end{center}

The elliptic parameter is $x/t^4$. If $F$ is the class of the
$E_7$ fiber, then $P \cdot F = 3$, while $F \cdot C = 2$, where $C$ is
the remaining component of the $A_7$ fiber not included in $F$. Since
these numbers are coprime, the genus $1$ fibration defined by $F$ has
a section. We therefore convert to the Jacobian, and read out the
Igusa--Clebsch invariants, which are as follows.

\begin{equation} \label{igcl25}
\begin{aligned}
I_2 &= 2 (4 r^2 s^2+20 r s^2+s^2-12 r^2 s+22 r s-6 s+r^2-6 r+3), \\
I_4 &= (16 r^4 s^4-32 r^3 s^4+24 r^2 s^4-8 r s^4+s^4+96 r^4 s^3-976 r^3 s^3 +904 r^2 s^3\\
& \qquad     -220 r s^3+56 r^4 s^2-392 r^3 s^2+198 r^2 s^2-24 r^4 s +20 r^3 s+r^4)/4, \\
I_6 &= (64 r^6 s^6+320 r^5 s^6-784 r^4 s^6+608 r^3 s^6-196 r^2 s^6+20 r s^6+s^6 +64 r^6 s^5 \\
& \qquad    -4128 r^5 s^5-17120 r^4 s^5+18544 r^3 s^5-3692 r^2 s^5 -378 r s^5-8 s^5\\
& \qquad      -1616 r^6 s^4+10304 r^5 s^4-35208 r^4 s^4  +30368 r^3 s^4-10769 r^2 s^4\\
& \qquad        +1720 r s^4+4 s^4-896 r^6 s^3 +6192 r^5 s^3-7152 r^4 s^3-1732 r^3 s^3\\
& \qquad     +3792 r^2 s^3-880 r s^3 +364 r^6 s^2-2228 r^5 s^2+4543 r^4 s^2-3312 r^3 s^2\\
& \qquad      +792 r^2 s^2 -36 r^6 s+238 r^5 s-264 r^4 s+80 r^3 s+r^6-8 r^5+4 r^4)/8 ,\\
I_{10} &= -16 r^7 (2 r-1)^2 s^5 .
\end{aligned}
\end{equation}

Computing the double cover, we obtain the following description of the
moduli space.
\begin{theorem}
A birational model for the surface $\widetilde{\sL}_5$ (equivalently,
for $Y_{-}(25)$) is given by
\[
z^2 = (4s^2+22s-1)^2  r^2 -2s(8s^3-88s^2+72s-17) r + s^2(2s-1)^2.
\]
It is a rational surface. The Humbert surface is birational to the
$(r,s)$-plane. In these coordinates, the Igusa--Clebsch invariants of a
point on the moduli space are given by the formulas in \eqref{igcl25}
above.
\end{theorem}

Once again, this is a conic bundle with a section. Setting $z = rm +
s(2s-1)$ and solving for $r$, we obtain
\[
r = -\frac{2 s (2 s m-m+8 s^3-88 s^2+72 s-17)}{(m-4 s^2-22 s+1) (m+4 s^2+22 s-1)}.
\]

\subsection{Comparison with previous work}

Let us compare our results with those of \cite{MSV}. There it is
asserted that the moduli space $\sL_5$ of genus $2$ curves with a
degree $5$ elliptic subfield (which is birational to the Humbert
surface $\sH_{25}$) is birational to the hypersurface cut out by the
following equation in $\mathbb{A}^3_{u,v,w}$:
\begin{align*}
0 &= 64 v^2 (u-4 v+1)^2 w^2 + 4 v (128 v^4-16 u^2 v^3+288 u v^3-2592 v^3-24 u^3 v^2+96 u^2 v^2+272 u v^2+4672 v^2 \\
& \qquad   +15 u^4 v-92 u^3 v+20 u^2 v-576 u v-2 u^5+12 u^4-4 u^3) w + (u^2+4 v u+4 v^2-48 v)^3.
\end{align*}
In fact, this is a rational surface: the discriminant with respect to $w$ is
\[
-(v+2 u-16) (16 v^2-u^2 v+36 u v+108 v-2 u^3)
\]
and the invertible change of variables 
\[
u = (v_1 + 25) h - 18, \qquad v = -2(v_1 + 25) h + (v_1 + 52)
\]
converts it to $(h^2 v_1+24 h^2-8 h-16)/v_1$, up to squares. Setting
this quantity equal to $g^2$ and solving for $v_1$, we obtain the full
parametrization
\begin{align*}
u &= (h^3-10 h^2-25 g^2 h+16 h+18 g^2)/(h^2-g^2), \\
v &= -2 (h^3-6 h^2-25 g^2 h+12 h+26 g^2-8)/(h^2-g^2), \\
w &= \frac{(h-5 g+4) (h+5 g+4)^2 (3 h^2-15 g h-10 h+14 g+8)^3}{16 (h-1) (h-g)^2 (3 h-15 g-8)^2 (h^3-6 h^2-25 g^2 h+12 h+26 g^2-8)}.
\end{align*}

\begin{remark}
  Unfortunately, there is a mistake\footnote{We take this opportunity
    to point out another typo: the last term for $a_0$ in equation
    (10) should be $12za$ rather than $12ya$.} in the paper
  \cite{MSV}: the moduli space described there, claimed to be
  birational to $\sL_5$ in Theorem $3$ of their paper, is actually
  birational to the Hilbert modular surface (i.e.~to
  $\widetilde{\sL}_5$), and the Humbert surface is a quotient of it by
  an involution which we give explicitly in the auxiliary files (we
  check directly that the absolute Igusa--Clebsch invariants are not
  changed by applying this involution).
\end{remark}

By comparing the Igusa--Clebsch invariants, we obtain the relation
between these coordinates and ours:
\begin{align*}
s &=(h+g) (3 h-15 g-8)/\big(4 (h^2-3 g h-10 g^2-2 g+4)\big), \\
m &= (73 h^5-571 g h^4-208 h^4+310 g^2 h^3+644 g h^3+312 h^3+3350 g^3 h^2+2220 g^2 h^2-856 g h^2 \\
& \qquad  -768 h^2+625 g^4 h-1700 g^3 h-2120 g^2 h+64 g h +192 h+3125 g^5+2500 g^4-5400 g^3\\
& \qquad  -4160 g^2+1344 g+1024 )/\big(4 (h-g) (h^2-3 g h-10 g^2-2 g+4)^2\big)
\end{align*}
with inverse
\begin{align*}
g &= \frac{s m^2+8 s^3 m-68 s^2 m+2 s m+8 m+16 s^5-272 s^4+1228 s^3-212 s^2-127 s -8}{(m+4 s^2-28 s-1) (2 s m-m+8 s^3-88 s^2+72 s-17)}, \\
h &= \frac{s m^2+8 s^3 m-124 s^2 m+74 s m-8 m+16 s^5-496 s^4+3484 s^3-3036 s^2 +649 s+8}{(m+4 s^2-28 s-1) (2 s m-m+8 s^3-88 s^2+72 s-17)}.
\end{align*}

\subsection{Elliptic curves}

The $j$-invariants of the associated elliptic curves are given by
\begin{align*}
j_1 &=  -(m^4+16 s^2 m^3-8 s m^3+20 m^3+96 s^4 m^2-8288 s^3 m^2+10632 s^2 m^2 \\
 & \qquad -5304 s m^2+1014 m^2+256 s^6 m-65920 s^5 m+116864 s^4 m-104416 s^3 m \\
 &  \qquad  +49824 s^2 m-24792 s m+5684 m+256 s^8-131584 s^7+4492928 s^6 \\
 &  \qquad  -5084160 s^5+1924400 s^4-243456 s^3-87256 s^2+14552 s+4945)^3 \\
& \qquad  /\big(4 (m+4 s^2+22 s-1)^5 (2 s m-m+8 s^3-88 s^2+72 s-17)^3 \\
& \qquad \quad  (m^2+8 s^2 m-4 s m+16 s^4-528 s^3-188 s^2-24 s-1) \big), \\
j_2 &= -(m^4+16 s^2 m^3-88 s m^3+20 m^3+96 s^4 m^2-1568 s^3 m^2+2632 s^2 m^2 \\
&  \qquad    -1064 s m^2+54 m^2+256 s^6 m-8320 s^5 m+57984 s^4 m-33696 s^3 m \\
& \qquad   +6944 s^2 m+6328 s m-1996 m+256 s^8-13824 s^7+207488 s^6-916480 s^5 \\
& \qquad   +510000 s^4-115456 s^3-53016 s^2+19592 s+2065)^3 \\
& \qquad   /\big(8 (m-4 s^2-22 s+1)^5 (2 s m-m+8 s^3-88 s^2+72 s-17)^2 \\
 & \qquad \quad (m^2+8 s^2 m-4 s m+16 s^4-528 s^3-188 s^2-24 s-1)\big) .
\end{align*}

Let us check that the $j$-invariant agrees with that implicit in
\cite{MSV}. It follows from the description there that $C$ maps to the
elliptic curve
\[
y^2 = x (x - 1) (x - \lambda)
\]
where $\lambda = \zeta F_1(\zeta)^2/F_2(\zeta)^2$ where $F_1$ and $F_2$ are the polynomials
\begin{align*}
F_1(T) &= T^2 + (2a + 2b + a^2)T + 2ab + b^2 \\
F_2(T) &= (2a + 1)T^2 + (a^2 + 2ab + 2b) T + b^2,
\end{align*}
$\zeta$ is a root of 
\[
(2a+1) \zeta^2 + (2b - 2ab - 2a - a^2)\zeta + b^2 + 2ab = 0
\]
and $u,v$ are related to the parameters $a,b$ by 
\[
u = \frac{2a(ab + b^2 + b +a + 1)}{b(a+b + 1)}, \quad v = \frac{a^3}{b(a+b+1)}.
\]
After some computation, the value of the $j$-invariant is found to
agree with $j_1$.

\subsection{Special loci}
We again list several interesting curves on the surface.
\begin{enumerate}
\item The curves $r = 0, z = \pm s(2s-1)$, $s =0, z = \pm r$ are
  rational curves corresponding to $I_{10} = 0$, as is the lift of $r
  = 1/2$.
\item The curve $4r^2s^2-4rs^2+s^2+6r^2s-2rs-2r^3+r^2 = 0$ on the
  Humbert surface, parametrized by $(r,s) = \big( (t^2-2t+4)/8,
  (t-2)(t^2-2t+4)/(4t^2) \big)$, lifts to an elliptic curve on
  $Y_{-}(25)$. It parametrizes the locus for which the two associated
  elliptic curves are related by a $14$-isogeny. Abstractly, this
  elliptic curve is isomorphic to $y^2 + xy + y = x^3 + 4x - 6$ of
  conductor $14$, which has rank $0$.
\item The curve $(2r - 1)^2s+r = 0$ lifts to an elliptic curve on
  $Y_{-}(25)$, isomorphic to the elliptic curve $y^2 + y = x^3 - x^2 -
  10x - 20$ of conductor $11$ and rank $0$. It parametrizes the locus
  of reducible Jacobians for which the two associated elliptic curves
  have an $11$-isogeny.
\item The curve 
\[
4(2r - 1)^2 s^4+ 4(44r^2+44r-1) s^3+(476r^2-144r+1)s^2+(-44r^2+34r)s+r^2 = 0
\]
of genus $0$ is part of the branch locus, and is parametrized by
$$(r,s) = \big( (t+1)^2(2t^2+4t+1)/(2(t^2+t-1)^2), -(2t^2+4t+1)/(2t^2)
\big).$$ It corresponds to the two elliptic curves being isomorphic.
\item The curve parametrized by
\[
r = \frac{1}{50t}, \quad s = \frac{(9t-1)(25t-1)}{16(25t^2+6t+1)}, \quad
m = \frac{25(37375 t^4+7044 t^3+762 t^2-124 t-1)}{64(25t^2+6t+1)^2}
\]
corresponds to the degenerate Case I of \cite{MSV}.
\item The curve parametrized by 
\[
r = \frac{-3(25t-1)}{2(25t-16) }, \quad
s =  \frac{1-t}{2}, \quad
m = \frac{15625t^3+3750t^2-7800t+89}{(t-1)(25t-16)^2}
\]
corresponds to the degenerate Case II of \cite{MSV}.
\item The curve parametrized by 
\[
r = \frac{1}{2(t-7)},\quad
s = \frac{3}{8}, \quad
m = -\frac{125}{16}
\]
corresponds to the degenerate case III of \cite{MSV}.
\end{enumerate}

\section{Discriminant $36$}

\subsection{Parametrization}

Next, we describe the moduli space for degree $6$ covers. The
corresponding Humbert surface is birational to the moduli space of
elliptic K3 surfaces with $E_8$, $A_5$ and $A_2$ fibers, with a
section of height $2 = 4 - 2/3 - 2 \cdot 4/6$. We briefly describe how
to parametrize this moduli space and the family of K3 surfaces over
it.

Start with an elliptic surface
\[
y^2 = x^3 + ax^2 + 2bt^2(t-1) x + ct^4(t-1)^2
\]
where $a,b,c$ are polynomials of degrees $2,1,1$ respectively. This
surface already has an $E_8$ fiber at $\infty$, an $A_3$ fiber at $0$
and an $A_1$ fiber at $t = 1$. Since the components of the eventual
$A_5$ fiber at $t = 0$ are rational, the constant term of $a$ is a
square, and we may scale $x$ and $y$ so that $a(0) = 1$. Similarly,
$a(1)$ is rational. So we may write 
\[
a = 1 - t + t e^2 + a_1 t(1-t).
\]
Imposing two more orders of vanishing at $t = 0$, we may solve linear
equations for $c_0$ and $c_1$. We need one more order of vanishing of
the discriminant at $t = 1$. This gives a quadratic equation in $b_1$
whose discriminant is $-a_1$ times a square. We set $a_1 = -f^2$ and
solve for $b_1$. Finally, we need a section with the requisite
height. It must have $x$-coordinate $t^2 (t-1)$ times a linear factor,
which we may normalize as $x_0 + 2b_0 + g^2 t$. Substituting in to the
Weierstrass cubic polynomial and completing the square, we can compute
the $y$-coordinate. This leads to three equations involving the
remaining variables. We eliminate $x_0$ and $b_0$ to get a single
equation connecting the remaining variables $e,f,g$. This equation is
cubic in $f$, but subtituting $e = e' - f$ and $g = g'f$ makes it
quadratic. Setting the discriminant equal to a square, after some easy
algebra we end up with a parametrization by two variables $r$ and
$s$. The universal Weierstrass equation is
\begin{align*}
y^2 &= x^3 + \Big(9(r+6)^2s^2(2s+r)^2/4t^2 -s(36s-r^2)(18rs^2+108s^2+9r^2s \\
 & \qquad \qquad +36rs+108s+r^3+9r^2+54r)/3t + (s-1)^2(36s-r^2)^2 \Big)x^2 \\
& \qquad + r(r+6)^2s^3(6s+r)t^2(rt-12) \cdot \Big(s(18rs^2+108s^2+9r^2s \\
& \qquad \qquad  +108s+2r^3+9r^2+54r)t-6(s-1)^2(36s-r^2) \Big)x \\
& \qquad + 3r^2(r+6)^4s^6(6s+r)^2t^4(rt-12)^2(rst+3s^2-6s+3).
\end{align*}
Here we have introduced appropriate scalings of $x,y,t$ to eliminate
denominators. The $x$-coordinate of the section of height $2$ is given by
\[
x_s = 3rs^3t^2(rt-12)(3st-6s-r+6).
\]

\subsection{Map to $\A_2$ and equation of $\widetilde{\sL}_6$}

The $3$-neighbor step corresponding to the figure below takes us to an
$E_8 E_7$ fibration.

\begin{center}
\begin{tikzpicture}

\draw (0,0)--(9,0)--(9.5,0.866)--(10.5,0.866)--(11,0)--(10.5,-0.866)--(9.5,-0.866)--(9,0);
\draw (8,0)--(8,1)--(8.5,1.866)--(7.5,1.866)--(8,1);
\draw (2,0)--(2,1);
\draw (7,3)--(7.5,1.866);
\draw (7,3) to [bend right] (7,0); 
\draw (7,3) to [bend left] (10.5,0.866); 
\draw [very thick] (7.5,1.866)--(8,1)--(8,0)--(9,0)--(9.5,-0.866)--(10.5,-0.866)--(11,0);
\draw [very thick] (9,0)--(9.5,0.866);

\fill [white] (0,0) circle (0.1);
\fill [white] (1,0) circle (0.1);
\fill [white] (2,0) circle (0.1);
\fill [white] (2,1) circle (0.1);
\fill [white] (3,0) circle (0.1);
\fill [white] (4,0) circle (0.1);
\fill [white] (5,0) circle (0.1);
\fill [white] (6,0) circle (0.1);
\fill [white] (7,0) circle (0.1);
\fill [black] (8,0) circle (0.1);
\fill [black] (8,1) circle (0.1);
\fill [white] (8.5,1.866) circle (0.1);
\fill [black] (7.5,1.866) circle (0.1);
\fill [black] (9,0) circle (0.1);
\fill [black] (9.5,0.866) circle (0.1);
\fill [black] (9.5,-0.866) circle (0.1);
\fill [white] (10.5,0.866) circle (0.1);
\fill [black] (10.5,-0.866) circle (0.1);
\fill [black] (11,0) circle (0.1);
\fill [white] (7,3) circle (0.1);

\draw (0,0) circle (0.1);
\draw (1,0) circle (0.1);
\draw (2,0) circle (0.1);
\draw (2,1) circle (0.1);
\draw (3,0) circle (0.1);
\draw (4,0) circle (0.1);
\draw (5,0) circle (0.1);
\draw (6,0) circle (0.1);
\draw (7,0) circle (0.1);
\draw (8,0) circle (0.1);
\draw (8,1) circle (0.1);
\draw (8.5,1.866) circle (0.1);
\draw (7.5,1.866) circle (0.1);
\draw (9,0) circle (0.1);
\draw (9.5,0.866) circle (0.1);
\draw (9.5,-0.866) circle (0.1);
\draw (10.5,0.866) circle (0.1);
\draw (10.5,-0.866) circle (0.1);
\draw (11,0) circle (0.1);
\draw (7,3) circle (0.1);

\end{tikzpicture}
\end{center}

We may now read out the Igusa--Clebsch invariants. They are fairly
lengthy to write down, so we will not display them here. Instead they
may be accessed in the auxiliary computer files. 
\begin{theorem}
A birational model for the surface $\widetilde{\sL}_6$ (equivalently,
for $Y_{-}(36)$) is given by
\begin{align*}
z^2 &= (36 s-r^2) \big( 1296 (r+6)^2 s^3-36 r (r+6) (r^2+162 r-648) s^2 + 12 r^2 (11 r^3+207 r^2+2592 r+8748) s \\
& \qquad \qquad \qquad + r^4 (7 r^2+108 r+972) \big).
\end{align*}
It is a singular elliptic K3 surface (i.e.~of maximal Picard number
$20$). The Humbert surface is birational to the $(r,s)$-plane.
\end{theorem}

The equation above shows that $\widetilde{\sL}_6$ has a genus $1$
fibration over $\Proj^1_r$, with an obvious section. Therefore it is
an elliptic surface. Taking the corresponding Jacobian fibration and
after some simple transformations (including setting $r = 4t$), we
obtain the following Weierstrass equation, showing that it is a K3
surface.
\begin{align*}
y^2 &= x^3 + 3 (t^2-4 t-1) (3 t^2-4 t-3) x^2 \\
 & \qquad + 96 t^2 (t+1) (3 t+1) (t^2-6 t-3) x \\
 & \qquad -256 t^3 (3 t+1)^2 (7 t^3-30 t^2-33 t-8).
\end{align*}
It has reducible fibers of type $\I_3$ at $t = 0$ and at the roots of
$t^3 - 3t^2 - 3t -1$, $\I_4$ at $t = -1/3$ and at $t = \infty$, and
$\I_2$ at the roots of $t^2 - 6t - 3$. The trivial lattice therefore
has rank $2 + 3 + 2 + 6 + 3 + 2 = 18$, leaving room for Mordell--Weil
rank up to $2$. We find the sections
\begin{align*}
P &= \big( 8(3t+1)(t^2-2t-1), 8(3t+1)^2(t^3-3t^2-3t-1)  \big) \\
Q &= \big( -4t(t-1)(3t+1), 4t(3t+1)^2(t^2-6t-3) \big).
\end{align*}
The intersection matrix of these sections is 
\[
\frac{1}{12} \begin{pmatrix}
3 & 3 \\ 3 & 7
\end{pmatrix}
\]
It has determinant $1/12$. Hence the span of these two sections and the
trivial lattice has rank $20$ and discriminant $3^4 \cdot 4^2 \cdot
2^2 \cdot (1/12) = 2^4 \cdot 3^3$.  We check that it is $2$- and $3$-
saturated in the N\'eron--Severi lattice, and therefore it must be all
of $\NS(Y_{-}(36))$.

\subsection{Elliptic curves}

We compute the symmetric polynomials in the two $j$-invariants, in
terms of the parameters on the Humbert surface, and check that the
discriminant of the polynomial $T^2 - (j_1 + j_2)T + j_1 j_2$ is, up
to squares, the same as the branch locus for the Hilbert modular
surface. We obtain the following expressions.
\begin{align*}
j_1 + j_2 &=  -54\big(1119744(r+6)^5s^9 -31104r(r+6)^4(r^2+162r-1944)s^8 \\
& \qquad \quad + 31104r^2(r+6)^3(4r^3+51r^2-3888r+43740)s^7 \\
& \qquad \quad -7776r^3(r+6)^2(5r^4-726r^3+9540r^2+36936r-2099520)s^6 \\
& \qquad \quad -2592r^4(r+6)(31r^5-777r^4-756r^3+208008r^2-5458752r-42515280)s^5 \\
& \qquad \quad -216r^5(301r^6+1002r^5-32472r^4-1228608r^3-59731344r^2-629226144r-1836660096)s^4 \\
& \qquad \quad -216r^6(50r^6-8601r^5-256410r^4-3320352r^3-32122656r^2-195570288r-459165024)s^3 \\
& \qquad \quad -54r^8(5r+54)(143r^4+4248r^3+61776r^2+419904r+944784)s^2 \\
& \qquad \quad + 6r^{11}(5r+54)^2(11r+162)s+ r^{12}(5r+54)^3 \big)/\big(r^{11}(6s+r)^4\big), \\
j_1 j_2 &= 729 \big( 144 (r+6)^2 s^4 -288 r (r+6) (r+114) s^3 + 72 r^2 (11 r^2+108 r+540) s^2 \\
& \qquad \quad + 24 r^3 (r+18) (5 r+54) s + r^4 (5 r+54)^2 \big)^3/\big(r^{12} (6 s+r)^6\big). 
\end{align*}

\subsection{Special loci}
We next present some curves of low genus on the modular surface.
\begin{enumerate}
\item The curves $r = 0$ and $r = -6s$ lift to pairs of rational
  curves on $Y_{-}(36)$. They correspond to $I_{10} =0$.
\item The curve $s = r^2/36$ lies in the branch locus. It is also a
  component of the product locus (i.e.~the Jacobian degenerates to a
  product of elliptic curves).
\item The curve $r = -6$ also belongs to the product locus. It lifts
  to a genus $0$ curve without any real points.
\item The curve $s = 0$ lifts to a genus $0$ curve whose points have
  $j_1 = j_2$. It too does not have any real points.
\item The curve $s = 1$ lifts to an elliptic curve isomorphic to $y^2
  + y = x^3 - x^2 - 10 x - 20$, of conductor $11$ and rank $0$. Its
  points correspond to reducible Jacobians such that the elliptic
  curves are related by an $11$-isogeny.
\item The curve $2rs+12s-r^2+2r = 0$ lifts to an elliptic curve
  isomorphic to $y^2 = x^3 + 22x^2 + 125x$, which has conductor $20$
  and rank $0$. Along it, the elliptic curves are related by a
  $5$-isogeny.
\item The curve $12s^2+6rs+r^2+2r= 0$ lifts to an elliptic curve
  isomorphic to $y^2 = x^3 + 30x^2 + 289x$, which has conductor $17$
  and rank $0$. It corresponds to elliptic curves linked by a
  $17$-isogeny.
\item The curve 
\begin{align*}
& \big( 1296 (r+6)^2 s^3-36 r (r+6) (r^2+162 r-648) s^2 + 12 r^2 (11 r^3+207 r^2+2592 r+8748) s  \\
& \qquad + r^4 (7 r^2+108 r+972) \big) = 0
\end{align*}
is part of the branch locus. It has genus $0$ and can be parametrized
as 
\[
r = \frac{-6(t+2)(t-1)^2}{t(t^2+2t-1)}, \qquad
s = \frac{(1+t+t^2)(t+2)^2(t-1)^2}{(t^2+t-1)(t^2+2t-1)t^2}
\]
It corresponds to an isogeny of degree $1$, i.e.~$j_1 = j_2$.
\item The curve 
\begin{align*}
& 46656(r+6)s^4  -216r^2(r^2+12r+180)s^3 + 36r^3(7r^2+96r-108)s^2 + 72r^5(r+12) \\
& \quad + r^6(5r+54) = 0
\end{align*}
on the Humbert surface has genus $0$, and can be parametrized by 
\[
r = \frac{6(3t^2-3t+1)^2}{t(t^3-5t^2+4t-1)}, \qquad 
s = \frac{(3t-1)(3t^2-3t+1)^3}{(t-1)t^3(t^3-5t^2+4t-1)}
\]
It lifts to an elliptic curve isomorphic to $y^2 + y = x^3 - 7$, which
has conductor $27$ and rank $0$. The points on this curve correspond
to the two elliptic curves being related by a $3$-isogeny.
\item The curve $r = -2s$ on the Humbert surface is non-modular. It
  lifts to an elliptic curve isomorphic to $y^2 = x^3 - 9x + 9$, which
  has conductor $324$ and rank $1$.
\item We can also produce non-modular curves of genus $0$ through the
  sections $P$ and $Q$ and their linear combinations. For instance,
  the section $P$ gives the curve 
\[
s = \frac{r^2(r^2+9r+162)}{54(r-36)(r+6)}, \quad z = \frac{-r^3(r^2-108r-972)(r^3-54r^2-972r-5832)}{27(r-36)^2(r+6)}
\]
while $P - Q$ gives the curve
\[
s = \frac{-r(7r^2+108r+1296)}{6(r^2-108r+216)}, \quad z = \frac{-2(r-36)r^2(r+6)(7r+12)(r^3-54r^2-972r-5832)}{(r^2-108r+216)^2}.
\]
The specializations of the tautological family along these genus $0$
curves are one-parameter families of genus $2$ curves with real
multiplication by $\sO_{36}$.
\end{enumerate}

\section{Discriminant $49$}

\subsection{Parametrization}

We start with the family of elliptic K3 surfaces having $E_8$, $A_4$,
$A_2$ and $A_1$ fibers at $t = \infty, 0, 1, h$ respectively, and a
section of height $49/30 = 4 - 6/5 - 2/3 - 1/2$.

We start with a Weierstrass equation 
\[
y^2 = x^3 + ax^2 + 2t(t-1)(t-h)bx + ct^2(t-h)^2(t-1)^2
\] 
with $a,b,c$ of degrees $2,1,1$ respectively.  We have shifted $x$ so
that the section has vanishing $x$-coordinate at $t = 0, 1, h$. 

This fibration already has an $E_8$ fiber at $\infty$ and $A_1$ fibers
at $t=0,1,h$. Let us normalize 
\begin{align*}
a &= (1-t) + e^2t + a_2 t (1-t) \\
b &= b_0 + b_1 t \\
c &= c_0 + c_1 t,
\end{align*}
where we have used that $a(0)$ and $a(1)$ must be squares in order for
the components to be defined over the base field, and also used the
Weierstrass scaling to fix $a(0) = 1$.

Imposing two more orders of vanishing at $t = 0$ gives linear
equations for $c_0$ and $c_1$, which are easily solved. A third order
of vanishing gives a quadratic equation for $b_1$, whose discriminant
equals $b_0(h+1) - a_2$, up to squares. We set it equal to $f^2$ and
solve for $a_2$ and then $b_1$. Next, we impose an extra order of
vanishing at $t = 1$, which gives a quadratic equation for $e$, with
discriminant up to squares equal to $f^2 - b_0$. We therefore set $b_0
= f^2 - g^2$ and solve for $e$. Next, scale $f = f' g$ and $m = m' g
(f'^2 - 1)$.  The section of height $49/30$ must have $x$-coordinate
$t(t-1)(t-h)(m^2t+g^2-f^2)$. Substituting this in to the Weierstrass
equation, we obtain a square times an expression quartic in
$t$. Completing the square, we obtain two equations. Solving these
leads to a rational moduli space with parameters $r$ and $s$.

The universal Weierstrass equation is as follows, where we have scaled
$t, x, y$ to avoid denominators:

\begin{align*}
y^2 &= x^3 - \Big( (4r^2s^2+4rs^2-s^2+4r^2s+2rs-r^2)  t^2 - 2(2r^3s^4\\
&\qquad \qquad +4r^3s^3 -4r^2s^3+rs^3+6r^3s^2-5r^2s^2+rs^2-s^2+4r^3s \\
&\qquad \qquad -7r^2s+2rs+2s-3r^2+2r) t - (rs^2+rs-s-r)^2 \Big)  x^2 \\
& \qquad + 8t \big( (s+1)t - rs - r - s \big) \big( rs t - (r-1)(rs-1) \big) \cdot \\
&\qquad \quad  \Big(  (4r^3s^4+8r^3s^3-6r^2s^3+rs^3+6r^3s^2-7r^2s^2+3rs^2 \\
& \qquad \qquad       -s^2+2r^3s -3r^2s+2rs-r^2)t  + (rs^2+rs-s-r)^2 \Big)  x \\
& \qquad  + 16t^2 ( (s+1)t - rs - r - s)^2 ( rs t - (r-1)(rs-1) )^2 \cdot \\
& \qquad \qquad \quad ( 4r^2s^2(s+1)(rs+r-1)  t + (rs^2+rs-s-r)^2 ).
\end{align*}

The $A_2$ fiber is now at $t = t_0 = (u-1)(uv-1)/(uv)$.

The $x$-coordinate of a section $P = (x_s, y_s)$ of height $49/30$ is
given by
\[
x_s = 4t \big( (s+1) t - rs - r - s \big) \big( rst - (r-1)(rs-1) \big) \big( rs(s+1) t - 1 \big).
\]
For brevity we omit the $y$-coordinate, though it may be found in the
auxiliary files.

\subsection{Map to $\A_2$ and equation of $\widetilde{\sL}_7$}

We now need to perform ``elliptic hopping'' to an $E_8 E_7$
fibration. We start with a $2$-neighbor step to a fibration with $E_8$
and $A_7$ fibers.

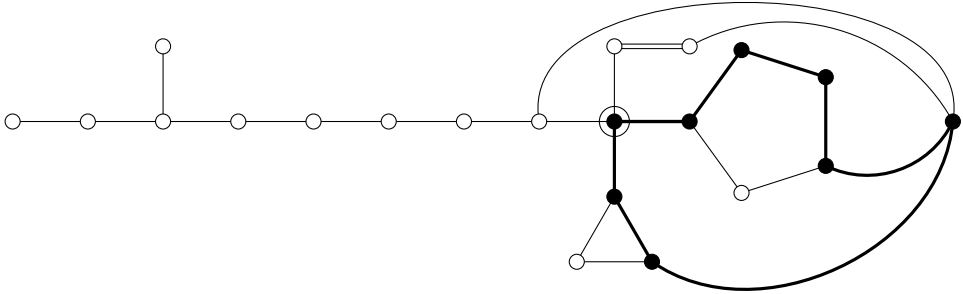
\begin{center}
\begin{tikzpicture}

\draw (-2,0)--(7,0)--(7.69,0.95)--(8.81,0.59)--(8.81,-0.59)--(7.69,-0.95)--(7,0);
\draw (6,1)--(6,-1)--(6.5,-1.866)--(5.5,-1.866)--(6,-1);
\draw (0,0)--(0,1);
\draw (6,0.97)--(7,0.97);
\draw (6,1.03)--(7,1.03);
\draw [very thick] (6,0)--(7,0)--(7.69,0.95)--(8.81,0.59)--(8.81,-0.59);
\draw [very thick] (6.5,-1.866)--(6,-1)--(6,0);
\draw [very thick] (6,0)--(7,0);

\draw (6,0) circle (0.2);
\draw (10.5,0) to [bend right=45] (7,1);
\draw (10.5,0) to [bend right=100] (5,0);
\draw [very thick] (10.5,0) to [bend left=45] (8.81,-0.59);
\draw [very thick] (10.5,0) to [bend left=60] (6.5,-1.866);

\fill [black] (7+0.69,0.95) circle (0.1);
\fill [white] (7+0.69,-0.95) circle (0.1);
\fill [black] (7+1.81,0.59) circle (0.1);
\fill [black] (7+1.81,-0.59) circle (0.1);

\fill [white] (6,1) circle (0.1);
\fill [white] (7,1) circle (0.1);
\fill [black] (6,-1) circle (0.1);
\fill [white] (5.5,-1.866) circle (0.1);
\fill [black] (6.5,-1.866) circle (0.1);

\fill [white] (-2,0) circle (0.1);
\fill [white] (-1,0) circle (0.1);
\fill [white] (0,0) circle (0.1);
\fill [white] (1,0) circle (0.1);
\fill [white] (2,0) circle (0.1);
\fill [white] (3,0) circle (0.1);
\fill [white] (4,0) circle (0.1);
\fill [white] (5,0) circle (0.1);
\fill [black] (6,0) circle (0.1);
\fill [white] (0,1) circle (0.1);
\fill [black] (7,0) circle (0.1);

\fill [black] (10.5,0) circle (0.1);

\draw (-2,0) circle (0.1);
\draw (-1,0) circle (0.1);
\draw (0,0) circle (0.1);
\draw (1,0) circle (0.1);
\draw (2,0) circle (0.1);
\draw (3,0) circle (0.1);
\draw (4,0) circle (0.1);
\draw (5,0) circle (0.1);
\draw (6,0) circle (0.1);
\draw (0,1) circle (0.1);
\draw (7,0) circle (0.1);

\draw (7+0.69,0.95) circle (0.1);
\draw (7+0.69,-0.95) circle (0.1);
\draw (7+1.81,0.59) circle (0.1);
\draw (7+1.81,-0.59) circle (0.1);

\draw (6,1) circle (0.1);
\draw (7,1) circle (0.1);
\draw (6,-1) circle (0.1);
\draw (5.5,-1.866) circle (0.1);
\draw (6.5,-1.866) circle (0.1);

\draw (10.5,0) circle (0.1);

\end{tikzpicture}
\end{center}

The elliptic parameter is given by 
\[
\frac{1}{t(t-t_0)} \left( \frac{y + y_s}{x-x_s} + \alpha_0 + \alpha_1 t \right)
\]
where 
\begin{align*}
\alpha_0 &= rs^2+(r-1)s-r, \\
\alpha_1 &= -\big(2r^2s^2+(2r^2-r+1)s+r^2-r \big)/(r-1).
\end{align*}

Converting to the Jacobian, it is then an easy matter to take a second
$2$-neighbor step to an $E_8 E_7$ fibration, as in the second step for
discriminant $25$. We omit the details, and also the formulas for the
Igusa--Clebsch invariants, which may be found in the auxiliary files.

We then compute the double cover defining the Hilbert modular surface.
\begin{theorem}
A birational model for the surface $\widetilde{\sL}_7$ (equivalently,
for $Y_{-}(49)$) is given by
\[ 
z^2 = -16 s^4 r^4 + 2s(20s^2+17s-1)r^3 -(44s^3+57s^2+18s-1)r^2 + 2s(15s+17)r + s^2.
\]
It is a singular elliptic K3 surface. The Humbert surface is
birational to the $(r,s)$-plane.
\end{theorem}

This surface has a genus $1$ fibration over $\Proj^1_s$, with a
section, for instance $(z,r) = (s, 0)$. Converting to the Jacobian
fibration, we get the Weierstrass equation
\[
y^2 = x^3 + (148s^3 + 45s^2 - 18s + 1)x^2 + 36s^4(20s^2+12s-1)x + 64s^7(1300s^2+885s-72)
\]
of an elliptic K3 surface. It has an $\I_7$ fiber at $s = 0$, an
$\I^*_0$ fiber at $s = \infty$, a $\IV$ fiber at $s = -1$, an $\I_3$
fiber at $s = 2/25$, and $ \I_2$ fibers at $s = (-5 \pm 2
\sqrt{7})/4$. So the contribution to the Picard number from the
trivial lattice is $18$. In addition we find the sections
\begin{align*}
P &= \big( 4s^3(4s - 3), 4s^3(s+1)(16s^2+40s-3) \big) \\
Q &= \big( -4s^2(5s-2) , 4s^2(s+1)(25s-2) \big)
\end{align*}
which are linearly independent. Therefore, the Picard number is $20$,
i.e.~the Hilbert modular surface is a singular K3 surface. The
intersection matrix of these sections is 
\[
\frac{1}{21} \begin{pmatrix}
13 & 4 \\ 4 & 5
\end{pmatrix}.
\]
It has discriminant $1/9$. Hence the span of these two sections and
the trivial lattice has rank $20$ and discriminant $7 \cdot 4 \cdot 3
\cdot 3 \cdot 2^2 \cdot (1/9) = 2^4 \cdot 7$. We check that it is
$2$-saturated in the N\'eron--Severi lattice, and therefore it must be
all of $\NS(Y_{-}(49))$.

Once again, we check that the symmetric polynomials of the two
$j$-invariants are rational functions of $r$ and $s$ (not displayed
here for brevity), and that the difference $j_1 - j_2$ generates the
function field of the Hilbert modular surface over that of the Humbert
surface. The equation of a tautological family of genus $2$ curves may
also be found in the computer files.

\subsection{Special loci}

\begin{enumerate}
\item The rational curves 
\begin{align*}
& s = -1, z = \pm (r-1)^2 \\
& s = 0, z = \pm r \\
& r = 0, z = \pm s \\
& s = \frac{1-r}{r}, z = \pm \frac{r^3+5r^2-8r+1}{r}
\end{align*}
correspond to $I_{10} = 0$.
\item The curve $rs = 1$ lifts to a rational curve on $Y_{-}(49)$. It
  is part of the product locus.
\item the curve $r = 1$ also lifts to a rational curve. It corresponds
  to the elliptic curves being related by a $13$-isogeny.
\item The curve $rs + r + s = 0$ lifts to an elliptic curve on the
  Hilbert modular surface, isomorphic to $y^2 + xy + y = x^3 - x^2 - x
  - 14$, of conductor $17$ and rank $0$. It corresponds to the
  elliptic curves being related by a $17$-isogeny.
\item The curve $s - r(s^2 + s - 1) = 0$ lifts to an elliptic curve on
  $Y_{-}(19)$, isomorphic to $y^2 + y = x^3 + x^2 - 9x - 15$ of
  conductor $19$ and rank $0$. It corresponds to the elliptic curves
  being related by a $19$-isogeny.
\item The curve $rs^3+2rs^2+s-r+2 = 0$ lifts to an elliptic curve
  isomorphic to $ y^2 = x^3 - x^2 - 4x + 4$, of conductor $24$ and
  rank $0$. It corresponds to the elliptic curves being
  $24$-isogenous.
\item The rational curve $4r^2s^2+4r^2s-2rs-s+r^2-r = 0$ on the
  Humbert surface can be parametrized by $(r,s) = \big( (t+1)/(t^2+1),
  (t^2-1)/4 \big)$. It lifts to an elliptic curve on $Y_{-}(49)$,
  isomorphic to $y^2 + y = x^3 - 7$, which has conductor $27$ and rank
  $0$. It corresponds to the elliptic curves being $27$-isogenous.
\item The rational curve $2rs^2+(2r-1)s+r-1 = 0$ lifts to an elliptic
  curve, isomorphic to $y^2 = x^3 + x^2 + 4x + 4$ of conductor $20$
  and rank $0$. It corresponds to the elliptic curves being
  $20$-isogenous.
\item The curve $ -16 s^4 r^4 + 2s(20s^2+17s-1)r^3
  -(44s^3+57s^2+18s-1)r^2 + 2s(15s+17)r + s^2 = 0$ has genus $0$ and
  is part of the branch locus. It can be parametrized as
\[
r = \frac{(t^2-5)(t^2-2t-7)}{4(t^2-4t-1)}, \qquad 
s = \frac{4(t^2-2t-7)}{(t^2-5)^2} 
\]
and corresponds to the two elliptic curves being isomorphic.
\item Again, we can obtain several genus $0$ non-modular curves from
  the sections. For instance, the section $P$ gives rise to 
\[
r = \frac{2s+3}{s(4s+3)}, \qquad z = -\frac{(s+1)(16s^2+40s-3)}{(4s+3)^2}
\]
and the section $Q$ to 
\[
r = 4/(5s), \qquad z = (s-2)(25s-2)/(25s).
\]
The tautological curve along this latter curve is short enough to note
here; it is given by
\begin{align*}
y^2 &= \big(x^3 + x^2(25s-2) -8x(s-4)(25s-2) -20(s-4)(12s+1)(25s-2) \big) \cdot \\
 &\qquad \big(x^3 -2x^2(25s-2)/5 -x(11s-142)(25s-2)/4 -5(25s-2)(3s^2+368s-148)/4 \big).
\end{align*}
\end{enumerate}

\section{Discriminant $64$}

\subsection{Parametrization}

We start with a family of elliptic K3 surfaces with reducible fibers
of types $E_7$, $A_3$ and $D_5$, and a section of height $2 = 4 - 3/4 - 5/4$.
The Weierstrass equation for this family is given by 
\begin{align*}
y^2 &= x^3 + x^2t\Big( -4(2s^5+6rs^4-s^4+4r^2s^3+4r^2s^2-4r^4)t  \\
    &\qquad  \qquad \qquad    + s^2(s+2r)(s^3+2rs^2+4s^2+8rs+8r^2) \Big) \\
    &\quad  + 2xt^3(t-1) \Big( -2s^3(s+2r)^2(s^2+2s+4r)(s^3+2s^2+12rs+8r^2)t \\
    &\qquad  \qquad \qquad \qquad + 2s^2(s+2r)^2 (3s^6+8rs^5-4s^5+8r^2s^4+12s^4 \\
    &\qquad \qquad \qquad \qquad +24r^2s^3+32rs^3+32r^3s^2 +32r^2s^2+32r^3s+32r^4) \Big) \\
    &\quad  + t^4(t-1)^2\Big( 64r(s-2)s^5(s+r)(s+2r)^4(s^2+2s+4r)t \\
    &\qquad \qquad \qquad \quad  + 16s^4(s+2r)^4(s^3-2s^2-4rs-4r^2)^2 \Big).
\end{align*}

The $x$-coordinate of a section of height $2$ is given by 
\[
x_s = 4(s^2+2s+4r)(t-1)t^2 \big( (s^2+2s+4r)t + (s-2r-2)(s+2r) \big).
\]
For brevity, we omit the description of the process of
parametrization, which is similar to that for the smaller
discriminants described in this paper. A brief outline is given in the
auxiliary files, for the interested reader.

\subsection{Map to $\A_2$ and equation of $\widetilde{\sL}_8$}

To go to an elliptic fibration with $E_8$ and $E_7$ fibers, we proceed
as follows. First, we take a $3$-neighbor step with fiber class as
shown below, to go to an elliptic fibration with $D_8$ and $E_7$
fibers.

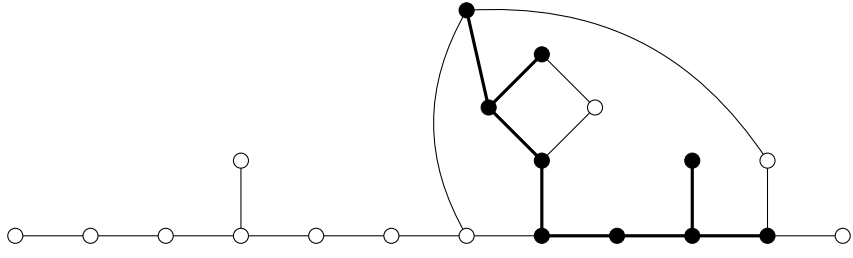
\begin{center}
\begin{tikzpicture}

\draw (0,0)--(11,0);
\draw (3,0)--(3,1);
\draw (7,0)--(7,1)--(7.707,1.707)--(7,2.414)--(6.293,1.707)--(7,1);
\draw (9,0)--(9,1);
\draw (10,0)--(10,1);
\draw (6,3)--(6.293,1.707);
\draw [bend right] (6,3) to (6,0);
\draw [bend left] (6,3) to (10,1);
\draw [very thick] (7,1)--(7,0)--(10,0);
\draw [very thick] (9,1)--(9,0);
\draw [very thick] (7,1)--(6.293,1.707)--(7,2.414);
\draw [very thick] (6,3)--(6.293,1.707);

\fill [white] (0,0) circle (0.1);
\fill [white] (1,0) circle (0.1);
\fill [white] (2,0) circle (0.1);
\fill [white] (3,0) circle (0.1);
\fill [white] (3,1) circle (0.1);
\fill [white] (4,0) circle (0.1);
\fill [white] (5,0) circle (0.1);
\fill [white] (6,0) circle (0.1);
\fill [black] (6,3) circle (0.1);
\fill [black] (7,0) circle (0.1);
\fill [black] (7,1) circle (0.1);
\fill [white] (7.707,1.707) circle (0.1);
\fill [black] (6.293,1.707) circle (0.1);
\fill [black] (7,2.414) circle (0.1);
\fill [black] (8,0) circle (0.1);
\fill [black] (9,0) circle (0.1);
\fill [black] (10,0) circle (0.1);
\fill [black] (9,1) circle (0.1);
\fill [white] (10,1) circle (0.1);
\fill [white] (11,0) circle (0.1);

\draw (0,0) circle (0.1);
\draw (1,0) circle (0.1);
\draw (2,0) circle (0.1);
\draw (3,0) circle (0.1);
\draw (3,1) circle (0.1);
\draw (4,0) circle (0.1);
\draw (5,0) circle (0.1);
\draw (6,0) circle (0.1);
\draw (6,3) circle (0.1);
\draw (7,0) circle (0.1);
\draw (7,1) circle (0.1);
\draw (7.707,1.707) circle (0.1);
\draw (6.293,1.707) circle (0.1);
\draw (7,2.414) circle (0.1);
\draw (8,0) circle (0.1);
\draw (9,0) circle (0.1);
\draw (10,0) circle (0.1);
\draw (9,1) circle (0.1);
\draw (10,1) circle (0.1);
\draw (11,0) circle (0.1);

\end{tikzpicture}
\end{center}

We then take a $2$-neighbor step to an $E_8 E_7$ fibration, as in the
calculation for discriminant $16$. We omit the details.

The equation for the double cover defining the Hilbert modular surface
can now be computed. 
\begin{theorem}
A birational model for the surface $\widetilde{\sL}_8$ (equivalently,
for $Y_{-}(64)$) is given by
\[
z^2 = -(s^3+2s^2+12rs+8r^2)(s+2r+2)(27s^4+54rs^3-52s^3-48rs^2+44s^2+96r^2s+72rs-16r^3-16r^2).
\]
Its minimal model is an honestly elliptic surface $X$ with
$\chi(\sO_X) = 3$ and Picard number $30$. The Humbert surface is
birational to the $(r,s)$-plane.
\end{theorem}

The equation above yields an elliptic surface: making the change of
variables $r = st$ and absorbing $s^4$ in $z^2$ makes the right side
quartic in $s$, giving an elliptic fibration over $\Proj^1_t$, with an
obvious section. Converting to the Jacobian form, after some simple
transformations, including the substitution $t = (v-1)/2$, we obtain
the Weierstrass equation
\[
y^2 = x^3+(v^6-14v^5+23v^4+108v^3+23v^2-14v+1)x^2 -16v^4(v^2-10v-7)(7v^2+10v-1)x
\]
which displays an obvious extra symmetry $(x,y,v) \to (x/v^6, y/v^9,
1/v)$.  This is an honestly elliptic surface with $\chi = 3$. It has
reducible fibers of type $\I_8$ at $\infty$ and $0$, $\I_2$ at $1$,
$\I_4$ at $-1$, $\I_2$ at the roots of $v^2 - 10v -7$ and also at the
roots of $7v^2 + 10v - 1$, and $\I_3$ at the roots of $v^2 - 10v + 1$.
The trivial sublattice has rank $28$, leaving room for Mordell--Weil
rank up to $2$. In addition to the $2$-torsion section $T = (0,0)$,
we find the following sections:
\begin{align*}
P &= \big( 4v(7v^2+10v-1), 4v(v+1)(v^2-10v+1)(7v^2+10v-1) \big), \\
Q &= \big( 4(v^2-10v-7),  12\sqrt{-3}(v-1)(v+1)^2(v^2-10v-7) \big).
\end{align*}
These sections are orthogonal of heights $1/6$ and $3/2$
respectively. Therefore, this elliptic surface has maximal Picard
number $30$. The discriminant of the lattice generated by $T, P, Q$
and the trivial lattice is $2^9 \cdot 3^2$. We checked that it is $2$-
and $3$-saturated. Therefore, it must be the full Picard group, and
the Mordell--Weil lattice is generated by these sections.

The Igusa--Clebsch invariants, the equation of a tautological family of
genus $2$ curves, and the elementary symmetric polynomials in the
$j$-invariants of the associated elliptic curves can be found in the
auxiliary files.

\subsection{Special loci}

\begin{enumerate}
\item The curve $r = 0$ is part of the product locus. It lifts to a
  genus $0$ curve on the Hilbert modular surface, with no real points.
\item The curve $s = 0, z = 16r^2(r+1)$ has genus $0$ and is part of
  the locus $I_{10} = 0$, as is the curve $s = 2$, which lifts to a
  rational curve on $\widetilde{\sL}_8$. The locus $I_{10} = 0$ also
  contains the curves $r = -s^2/4$ (which lifts to an elliptic curve),
  $r = -s^2/4 - s/2$ (lifts to a union of two rational curves), and
  $s^3+2s^2+12rs+8r^2 = 0$, which has genus $0$ and is part of the
  branch locus. The last curve can be parametrized as $(r,s) = \big(
  -t(1+3t+t^2), -2(t^2 + 3t+1) \big)$.
\item The curve $s^3+2rs^2-2s^2+4r^2 = 0$ has genus $0$ and can be
  parametrized as $(r,s) = \Big(-t(t^2-2)/\big(2(t+1)\big),
  -(t^2+2)/(t+1) \Big)$. It lifts to an elliptic curve of conductor
  $14$, and is modular (corresponding to the two elliptic curves being
  $7$-isogenous).
\item The rational curve $s = -2r-2$ is part of the branch locus, as
  is the curve
\[
27s^4+54rs^3-52s^3-48rs^2+44s^2+96r^2s+72rs-16r^3-16r^2 = 0.
\]
The latter can be parametrized as 
\[
(r,s) = \left( \frac{ (t^2+26t+124)(t^2+40t+408) }{216(t+20) }, \frac{(t+16)(t^2+26t+124) }{ 54(t+20)} \right)
\]
\item There are several non-modular specializations: for instance, $s
  \in \{-2,1/2, 5/2\}$ or $r = -1$ all yield elliptic curves with
  rational points. Of course, we may use the elliptic fibration to
  produce many elliptic curves, simply by specializing the parameter
  on the base.
\item The sections of the elliptic fibrations gives us several
  (usually) non-modular rational curves. For instance, $P$ gives the
  curve parametrized by
\[
r = -\frac{(v-1)(2v-1)}{v^2-v+1}, \quad 
s = -\frac{2(2v-1)}{v^2-v+1}, \quad
z = \frac{16(v-1)v^2(v+1)(2v-1)^2(v^2-10v+1)}{(v^2-v+1)^4},
\]
while
$Q$ gives the curve parametrized by
\begin{align*}
r &= -\frac{(v-1)(7v^2+4v-7)}{(v+1)(7v^2-4v+7)}, \quad
s = -\frac{2(7v^2+4v-7)}{(v+1)(7v^2-4v+7)}, \quad \\
z &= \frac{48\sqrt{-3}v^3(v^2-10v-7)(7v^2+4v-7)^2(7v^2+10v-1)}{(v+1)^4(7v^2-4v+7)^4}.
\end{align*}

\end{enumerate}

\section{Discriminant $81$}

\subsection{Parametrization}

We start with a family of elliptic K3 surfaces with reducible fibers
of types $E_6$, $A_2$ and $A_7$, and a section of height $9/8 = 4 -
4/3 - 2/3 - 7/8$. A Weierstrass equation is given by
\begin{align*}
y^2 &= x^3 + x^2\Big( (4r^2s^6+4rs^6+s^6-24r^2s^5-24rs^5-6s^5-36r^2s^4+108rs^4+15s^4 +240r^2s^3 +144rs^3\\
 & \qquad \qquad \quad -84s^3+252r^2s^2 -132rs^2-177s^2-216r^2s-504rs-198s+36r^2-108r-63)  t^2 \\
& \qquad \qquad \quad -2(2r^2s^3+rs^3+10r^2s^2+rs^2-2s^2-2r^2s-13rs-4s+6r^2-5r-2)  t + r^2 \Big)  \\
&\quad + 32 \, x \, (s+1)^2(rs-r-1)t^2 \big( 4(s^2+3)^2t -(s-1)(2rs+s+2r-1) \big) \cdot \\
& \qquad \quad \Big( (2rs^6+s^6-4rs^5-2s^5-16r^2s^4+38rs^4+3s^4+64r^2s^3 +24rs^3 -44s^3 \\
& \qquad \qquad+96r^2s^2-34rs^2-73s^2-192r^2s-276rs-114s+48r^2-6r-27)  t^2 \\
& \qquad \qquad -2(r^2s^3+rs^3+9r^2s^2+rs^2-2s^2-5r^2s-13rs-4s+3r^2-5r-2)  t + r^2 \Big) \\
& \quad + 256(s+1)^4(rs-r-1)^2t^4 \big( 4(s^2+3)^2t -(s-1)(2rs+s+2r-1) \big)^2  \cdot \\
& \qquad \quad \Big( (s-1)^2(s^2+2s+8r+5)^2  t^2 -2(rs^3+8r^2s^2+rs^2-2s^2-8r^2s-13rs-4s-5r-2)  t + r^2 \Big),
\end{align*}
and the $x$-coordinate of the extra section is 
\[
x_s =  -4(s+1)^2(rs-r-1)t \big( 4(s^2+3)^2 t -(s-1)(2rs+s+2r-1) \big)/(s-1)^2.
\]

\subsection{Map to $\A_2$ and equation of $\widetilde{\sL}_9$}

We first go to an elliptic fibration with $E_8$ and $A_7$ fibers, via
the $3$-neighbor step corresponding to the figure below.

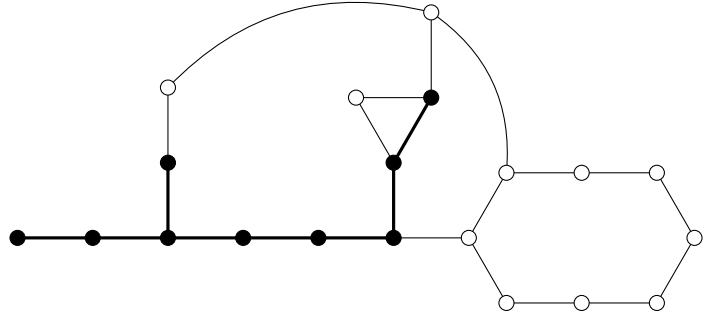
\begin{center}
\begin{tikzpicture}

\draw (0,0)--(6,0)--(6.5,0.866)--(8.5,0.866)--(9,0)--(8.5,-0.866)--(6.5,-0.866)--(6,0);
\draw (5,0)--(5,1)--(5.5,1.866)--(4.5,1.866)--(5,1);
\draw (2,0)--(2,2);
\draw (5.5,3)--(5.5,1.866);
\draw [bend right] (5.5,3) to (2,2);
\draw [bend left] (5.5,3) to (6.5,0.866);
\draw [very thick] (2,0)--(2,1);
\draw [very thick] (0,0)--(5,0)--(5,1)--(5.5,1.866);

\fill [black] (0,0) circle (0.1);
\fill [black] (1,0) circle (0.1);
\fill [black] (2,0) circle (0.1);
\fill [black] (2,1) circle (0.1);
\fill [white] (2,2) circle (0.1);
\fill [black] (3,0) circle (0.1);
\fill [black] (4,0) circle (0.1);
\fill [black] (5,0) circle (0.1);
\fill [black] (5,1) circle (0.1);
\fill [white] (4.5,1.866) circle (0.1);
\fill [black] (5.5,1.866) circle (0.1);
\fill [white] (5.5,3.0) circle (0.1);
\fill [white] (6,0) circle (0.1);
\fill [white] (6.5,0.866) circle (0.1);
\fill [white] (6.5,-0.866) circle (0.1);
\fill [white] (7.5,0.866) circle (0.1);
\fill [white] (7.5,-0.866) circle (0.1);
\fill [white] (8.5,0.866) circle (0.1);
\fill [white] (8.5,-0.866) circle (0.1);
\fill [white] (9,0) circle (0.1);

\draw (0,0) circle (0.1);
\draw (1,0) circle (0.1);
\draw (2,0) circle (0.1);
\draw (2,1) circle (0.1);
\draw (2,2) circle (0.1);
\draw (3,0) circle (0.1);
\draw (4,0) circle (0.1);
\draw (5,0) circle (0.1);
\draw (5,1) circle (0.1);
\draw (4.5,1.866) circle (0.1);
\draw (5.5,1.866) circle (0.1);
\draw (5.5,3.0) circle (0.1);
\draw (6,0) circle (0.1);
\draw (6.5,0.866) circle (0.1);
\draw (6.5,-0.866) circle (0.1);
\draw (7.5,0.866) circle (0.1);
\draw (7.5,-0.866) circle (0.1);
\draw (8.5,0.866) circle (0.1);
\draw (8.5,-0.866) circle (0.1);
\draw (9,0) circle (0.1);

\end{tikzpicture}
\end{center}

Finally, a $2$-neighbor step (as in the second step for discriminant
$25$) takes us to an $E_8 E_7$ fiber. We can then read out the
Igusa--Clebsch invariants, which are in the auxiliary files, along with
a tautological family of genus $2$ curves and the sum and product of
the $j$-invariants of the elliptic curves.

We then compute the double cover defining the Hilbert modular surface.
\begin{theorem}
A birational model for the surface $\widetilde{\sL}_9$ (equivalently,
for $Y_{-}(81)$) is given by
\begin{align*}
z^2 &= 16(s^3-9s^2-9s+9)^2  r^4  + 32s(s^5+5s^4+90s^3+18s^2-171s-135)  r^3 \\
    &\qquad  + 8(3s^6+40s^5+81s^4-336s^3-627s^2-360s-81)  r^2 \\
    &\qquad  + 8(s+1)^2(s^4+5s^3-49s^2-57s-12)  r  + (s+1)^4(s^2-18s-15).
\end{align*}
Its minimal model is an honestly elliptic surface $X$ with
$\chi(\sO_X) = 3$ and Picard number $29$. The Humbert surface is
birational to the $(r,s)$-plane.
\end{theorem}

The equation above evidently defines an elliptic surface over
$\Proj^1_s$ with section. Converting to the Jacobian form, we obtain
(after Weierstrass transformations and the change of variable $s =
t-1$)
\begin{align*}
y^2 &= x^3 + 4(44t^5+149t^4-184t^3+280t^2-128t+16)x^2 \\
  & \qquad +64t^5(125t^5+1403t^4-946t^3+1672t^2+472t-176)x \\
 & \qquad + 256t^9(13625t^5-3964t^4+18780t^3+3568t^2+17296t-4096)
\end{align*}
It is honestly elliptic with $\chi = 3$, and reducible fibers of type
$\I_8$ at $0$, $\I_4^*$ at $\infty$, $\I_3$ at the roots of
$t^2-2t+4$, $\I_2$ at the roots of $4t^3-12t^2+21t-4$ and $\I_3$ at
the roots of $25t^3-66t^2+84t-16$. These give rise to a trivial
lattice of rank $27$, leaving room for Mordell--Weil rank up to $3$.
We find the sections
\begin{align*}
P &= \big( 128t^3-256t^4, 16t^3(3t-4)(25t^3-66t^2+84t-16) \big) \\
Q &= \big( 16t^4(t^2-7t-2), 16t^4(t^2-2t+4)(4t^3-12t^2+21t-4) \big)
\end{align*}
with height matrix
\[
\frac{1}{18}\begin{pmatrix}
18 & -12 \\ -12 & 17
\end{pmatrix}
\]
Therefore the Picard number $\rho$ is at least $29$. On the other
hand, by counting points modulo $7$ and $13$ and using the method of
van Luijk \cite{vL} (comparing the square classes of the discriminants
of the Picard groups of these reductions) we obtain that the Picard
number cannot be $30$. Therefore $\rho = 29$ exactly. We checked that
the lattice generated by $P$ and $Q$, which has discriminant $2^4
\cdot 3^7$, is $2$- and $3$-saturated. Therefore it is the full
Mordell--Weil lattice.

\subsection{Special loci}

\begin{enumerate}
\item The curves $r = 0$, $s = -4r-1$ and $2rs^2-s^2-4rs-2s+2r-1 =0$
  of genus $0$ on the Humbert surface are all part of the product
  locus. The first lifts to a rational curve on $Y_{-}(81)$ and the
  other two lift to elliptic curves.
\item The curves $s = -1$, $rs -r-1 =0$ and $s^2 + 3 = 0$ are part of
  the locus $I_{10} = 0$. The first two lift to unions of two rational
  curves, while the last clearly has no real points.
\item The curve $s=1$ lifts to an elliptic curve isomorphic to
  $X_0(11)$. It corresponds to the two elliptic curves being
  $11$-isogenous.
\item The branch locus has genus $1$, and corresponds to $j_1 = j_2$.
\item The section $P$ gives rise to the following non-modular curve of
  genus $0$:
\begin{align*}
r &= -\frac{(s+1)^2(s^2-18s-27)}{(s-3)(7s^3+27s^2+45s+9)} \\
z &= \frac{(s-1)(s+1)^3(25s^3+9s^2+27s+27)(s^4-18s^3-144s^2-342s-729)}{(s-3)^2(7s^3+27s^2+45s+9)^2}.
\end{align*}
The section $Q$ gives the non-modular curve
\[
r = \frac{(s+1)^2}{2(s-3)s}, \qquad
z = -\frac{(s-1)(s+1)^3(4s^3+9s+9)}{(s-3)^2s^2}.
\]
The tautological curve over it is given by
\begin{align*}
y^2 &= \big(x^3 + 3x^2(s^2-3s+6)(4s^3+9s+9) \\
& \qquad -24x(s-3)s^2(s^2+3)(4s^2-3s+18)(4s^3+9s+9) \\
& \qquad  -4(s-3)s(4s^2-3s+18)^2(4s^3+9s+9)^2(5s^3-9s^2+9s-9) \big) \cdot\\
& \quad \big( x^3 -6x^2(s-3)(s^2+3)(4s^2-3s+18) \\
& \qquad -3x(s+3)(4s^2-3s+18)^2(5s^2+3)(4s^3+9s+9)/4 \\
& \qquad  + (4s^2-3s+18)^3(4s^3+9s+9) \cdot \\
& \qquad \qquad (100s^6-225s^5+945s^4-1062s^3+2106s^2-729s+81)/4 \big).
\end{align*}

\end{enumerate}

\section{Discriminant $100$}

\subsection{Parametrization}

We start with a family of elliptic K3 surfaces with $D_6$, $A_6$ and
$A_3$ fibers, and a section of height $25/12 = 4 - 6/4 - 3/4 - 6/7$.
A Weierstrass equation is given by 
\begin{align*}
y^2 &=  x^3 + x^2 \Big( -r^5(r+2)^3(2s^2+2s-r^2)  t^3 -r^2(r+2)^2(2s^6+6s^5-5r^2s^4 \\
& \qquad \qquad \qquad -2rs^4+3s^4-10r^2s^3+2r^4s^2-2r^2s^2 +2r^4s-4r^3s-r^4) t^2 \\
& \qquad \qquad \qquad + r(r+2)s(s+2)(s-r)(s+r)(s^4+2s^3-r^2s^2-2s^2+4rs+2r^2) t \\
& \qquad \qquad \qquad + s^2(s+2)^2(s-r)^2(s+r)^2 \Big)/\big(s^2(s+2)^2(s-r)^2(s+r)^2 \big) \\
& \quad + 2x r^3(r+2)^3(s+r+2)(t-1)t^2\Big( r^2(r+2)^2(s^2+s-r^2) t^2 \\
& \qquad \qquad -r(r+2)(s^4+2s^3-r^2s^2-s^2-r^2s+2rs+r^2) t \\
& \qquad \qquad -s(s+2)(s-r)(s+r) \Big)/\big(s(s+2)^3(s-r)^2(s+r)^3\big) \\
& \quad + r^6(r+2)^6(s+r+2)^2(t-1)^2t^4(r^2t+2rt+1)/\big((s+2)^4(s-r)^2(s+r)^4\big),
\end{align*}
and the extra section having $x$-coordinate
\[
x_s = s^3(s+2)(s-r)^2(s+r)(s+r+2)t(t-r^2-2r)/r^2.
\]

\subsection{Map to $\A_2$ and equation of $\widetilde{\sL}_{10}$}

We first go to an elliptic fibration with $E_7$, $A_7$ and $A_1$
fibers by a $2$-neighbor step.

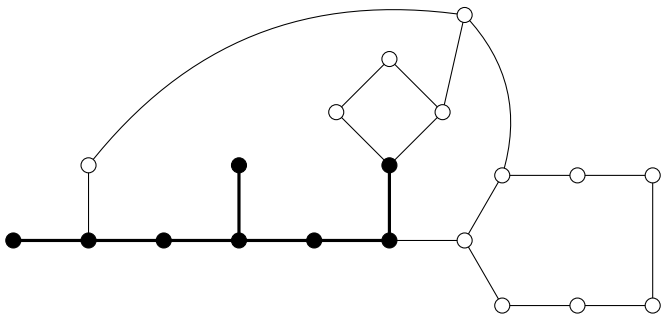
\begin{center}
\begin{tikzpicture}

\draw (0,0)--(6,0)--(6.5,0.866)--(8.5,0.866)--(8.5,-0.866)--(6.5,-0.866)--(6,0);
\draw (1,0)--(1,1);
\draw (3,0)--(3,1);
\draw (5,0)--(5,1)--(5.707,1.707)--(5,2.414)--(4.293,1.707)--(5,1);
\draw (6,3)--(5.707,1.707);
\draw [bend right] (6,3) to (1,1);
\draw [bend left] (6,3) to (6.5,0.866);
\draw [very thick] (0,0)--(5,0)--(5,1);
\draw [very thick] (3,0)--(3,1);

\fill [black] (0,0) circle (0.1);
\fill [black] (1,0) circle (0.1);
\fill [white] (1,1) circle (0.1);
\fill [black] (2,0) circle (0.1);
\fill [black] (3,0) circle (0.1);
\fill [black] (3,1) circle (0.1);
\fill [black] (4,0) circle (0.1);
\fill [black] (5,0) circle (0.1);
\fill [black] (5,1) circle (0.1);
\fill [white] (5.707,1.707) circle (0.1);
\fill [white] (4.293,1.707) circle (0.1);
\fill [white] (5,2.414) circle (0.1);
\fill [white] (6,0) circle (0.1);
\fill [white] (6,3) circle (0.1);
\fill [white] (6.5,0.866) circle (0.1);
\fill [white] (6.5,-0.866) circle (0.1);
\fill [white] (7.5,0.866) circle (0.1);
\fill [white] (7.5,-0.866) circle (0.1);
\fill [white] (8.5,0.866) circle (0.1);
\fill [white] (8.5,-0.866) circle (0.1);

\draw (0,0) circle (0.1);
\draw (1,0) circle (0.1);
\draw (1,1) circle (0.1);
\draw (2,0) circle (0.1);
\draw (3,0) circle (0.1);
\draw (3,1) circle (0.1);
\draw (4,0) circle (0.1);
\draw (5,0) circle (0.1);
\draw (5,1) circle (0.1);
\draw (5.707,1.707) circle (0.1);
\draw (4.293,1.707) circle (0.1);
\draw (5,2.414) circle (0.1);
\draw (6,0) circle (0.1);
\draw (6,3) circle (0.1);
\draw (6.5,0.866) circle (0.1);
\draw (6.5,-0.866) circle (0.1);
\draw (7.5,0.866) circle (0.1);
\draw (7.5,-0.866) circle (0.1);
\draw (8.5,0.866) circle (0.1);
\draw (8.5,-0.866) circle (0.1);

\end{tikzpicture}
\end{center}

Next, we go to a $D_6 D_5 A_3$ fibration, via the fiber class shown below.

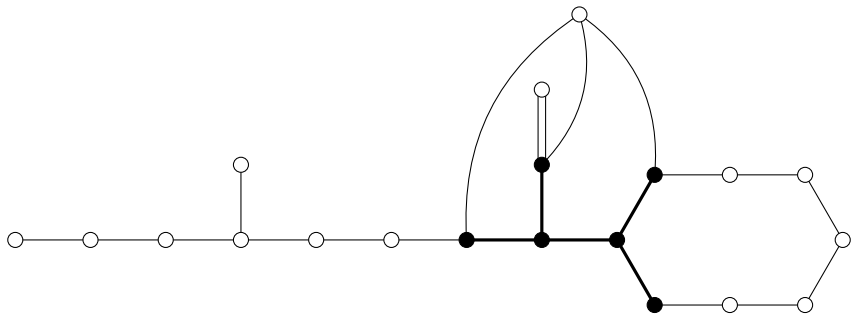
\begin{center}
\begin{tikzpicture}

\draw (0,0)--(8,0)--(8.5,0.866)--(10.5,0.866)--(11,0)--(10.5,-0.866)--(8.5,-0.866)--(8,0);
\draw (3,0)--(3,1);
\draw (7,0)--(7,1);
\draw (6.95,1)--(6.95,2);
\draw (7.05,1)--(7.05,2);
\draw [bend left] (7.5,3) to (7,1);
\draw [bend left] (7.5,3) to (8.5,0.866);
\draw [bend right] (7.5,3) to (6,0);
\draw [very thick] (6,0)--(8,0);
\draw [very thick] (7,0)--(7,1);
\draw [very thick] (8.5,0.866)--(8,0)--(8.5,-0.866);

\fill [white] (0,0) circle (0.1);
\fill [white] (1,0) circle (0.1);
\fill [white] (2,0) circle (0.1);
\fill [white] (3,0) circle (0.1);
\fill [white] (3,1) circle (0.1);
\fill [white] (4,0) circle (0.1);
\fill [white] (5,0) circle (0.1);
\fill [black] (6,0) circle (0.1);
\fill [black] (7,0) circle (0.1);
\fill [black] (7,1) circle (0.1);
\fill [white] (7,2) circle (0.1);
\fill [black] (8,0) circle (0.1);
\fill [white] (7.5,3) circle (0.1);
\fill [black] (8.5,0.866) circle (0.1);
\fill [black] (8.5,-0.866) circle (0.1);
\fill [white] (9.5,0.866) circle (0.1);
\fill [white] (9.5,-0.866) circle (0.1);
\fill [white] (10.5,0.866) circle (0.1);
\fill [white] (10.5,-0.866) circle (0.1);
\fill [white] (11,0) circle (0.1);

\draw (0,0) circle (0.1);
\draw (1,0) circle (0.1);
\draw (2,0) circle (0.1);
\draw (3,0) circle (0.1);
\draw (3,1) circle (0.1);
\draw (4,0) circle (0.1);
\draw (5,0) circle (0.1);
\draw (6,0) circle (0.1);
\draw (7,0) circle (0.1);
\draw (7,1) circle (0.1);
\draw (7,2) circle (0.1);
\draw (8,0) circle (0.1);
\draw (7.5,3) circle (0.1);
\draw (8.5,0.866) circle (0.1);
\draw (8.5,-0.866) circle (0.1);
\draw (9.5,0.866) circle (0.1);
\draw (9.5,-0.866) circle (0.1);
\draw (10.5,0.866) circle (0.1);
\draw (10.5,-0.866) circle (0.1);
\draw (11,0) circle (0.1);

\end{tikzpicture}
\end{center}

The new elliptic fibration has Mordell--Weil group of rank $2$,
generated by two sections $P$ and $Q$ with height matrix
\[
\frac{1}{4} 
\begin{pmatrix}
2 & 1 \\ 1 & 13
\end{pmatrix}
\]
with the entries arising as
\begin{align*}
1/2 &= 4 - 3/4 - 5/4 - 6/4, \\
13/4 &= 4 - 3/4, \\
1/4 &= 2 - 1 - 3/4.
\end{align*}

A $2$-neighbor step takes us to a fibration with $D_8 D_6 A_1$ fibers,
a $2$-torsion section $T$ (and a section of height $25/2$).

The section $Q$ intersects the fiber class in $3$, whereas the
remaining component of the $D_5$ fiber intersects it in $2$. Since
these are coprime, the new fibration has a section. Note also that the
section $2P$ of height $2 = 4 - 1 - 2 \cdot 2/4$ doesn't intersect the
fiber class, giving rise to an extra $A_1$ fiber.

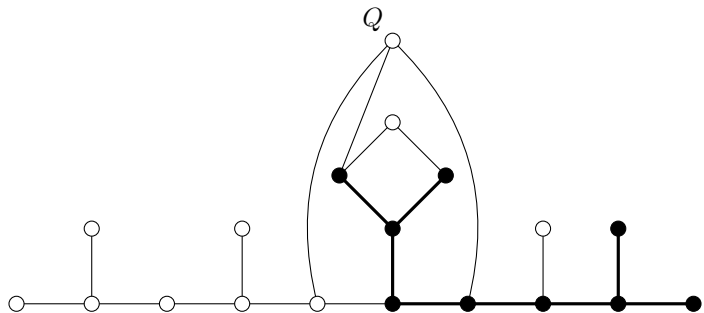
\begin{center}
\begin{tikzpicture}

\draw (0,0)--(9,0);
\draw (1,0)--(1,1);
\draw (3,0)--(3,1);
\draw (5,0)--(5,1)--(5.707,1.707)--(5,2.414)--(4.293,1.707)--(5,1);
\draw (7,0)--(7,1);
\draw (8,0)--(8,1);
\draw (5,3.5)--(4.293,1.707);
\draw [bend left] (5,3.5) to (6,0);
\draw [bend right] (5,3.5) to (4,0);
\draw [very thick] (5,1)--(5,0)--(9,0);
\draw [very thick] (8,0)--(8,1);
\draw [very thick] (5.707,1.707)--(5,1)--(4.293,1.707);

\fill [white] (0,0) circle (0.1);
\fill [white] (1,0) circle (0.1);
\fill [white] (1,1) circle (0.1);
\fill [white] (2,0) circle (0.1);
\fill [white] (3,0) circle (0.1);
\fill [white] (3,1) circle (0.1);
\fill [white] (4,0) circle (0.1);
\fill [white] (5,3.5) circle (0.1);
\fill [black] (5,0) circle (0.1);
\fill [black] (5,1) circle (0.1);
\fill [black] (5.707,1.707) circle (0.1);
\fill [black] (4.293,1.707) circle (0.1);
\fill [white] (5,2.414) circle (0.1);
\fill [black] (6,0) circle (0.1);
\fill [black] (7,0) circle (0.1);
\fill [black] (8,0) circle (0.1);
\fill [white] (7,1) circle (0.1);
\fill [black] (8,1) circle (0.1);
\fill [black] (9,0) circle (0.1);

\draw (0,0) circle (0.1);
\draw (1,0) circle (0.1);
\draw (1,1) circle (0.1);
\draw (2,0) circle (0.1);
\draw (3,0) circle (0.1);
\draw (3,1) circle (0.1);
\draw (4,0) circle (0.1);
\draw (5,3.5) circle (0.1);
\draw (5,0) circle (0.1);
\draw (5,1) circle (0.1);
\draw (5.707,1.707) circle (0.1);
\draw (4.293,1.707) circle (0.1);
\draw (5,2.414) circle (0.1);
\draw (6,0) circle (0.1);
\draw (7,0) circle (0.1);
\draw (8,0) circle (0.1);
\draw (7,1) circle (0.1);
\draw (8,1) circle (0.1);
\draw (9,0) circle (0.1);

\draw [above left] (5,3.5) node{$Q$};

\end{tikzpicture}
\end{center}

Finally, a $2$-neighbor step leads us to the desired $E_8 E_7$
fibration. The fiber $F$ with an extended $E_8$ Dynkin diagram is
shown below, and the $2$-torsion section $T$ combines with some of the
$D_6$ components and the non-identity component of the $A_1$ fiber to
give an $E_7$ diagram disjoint from $F$. Also, the section $P$ of
height $25/2$, or its negative, has odd intersection number with $F$,
while the remaining component of the $D_8$ fiber has intersection $2$
with $F$. So the new fibration has a section.  We can read out the
Igusa--Clebsch invariants, which are in the auxiliary files.

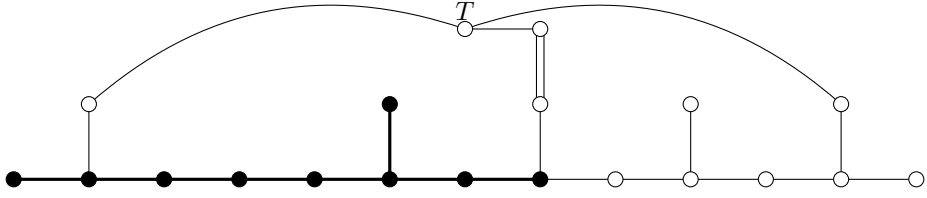
\begin{center}
\begin{tikzpicture}

\draw (0,0)--(12,0);
\draw (1,0)--(1,1);
\draw (5,0)--(5,1);
\draw (7,0)--(7,1);
\draw (9,0)--(9,1);
\draw (11,0)--(11,1);
\draw (7,2)--(6,2);
\draw (6.95,1)--(6.95,2);
\draw (7.05,1)--(7.05,2);

\draw [bend right] (6,2) to (1,1);
\draw [bend left] (6,2) to (11,1);
\draw [very thick] (0,0)--(7,0);
\draw [very thick] (5,0)--(5,1);

\fill [black] (0,0) circle (0.1);
\fill [black] (1,0) circle (0.1);
\fill [white] (1,1) circle (0.1);
\fill [black] (2,0) circle (0.1);
\fill [black] (3,0) circle (0.1);
\fill [black] (4,0) circle (0.1);
\fill [black] (5,0) circle (0.1);
\fill [black] (5,1) circle (0.1);
\fill [black] (6,0) circle (0.1);
\fill [black] (7,0) circle (0.1);
\fill [white] (7,1) circle (0.1);
\fill [white] (7,2) circle (0.1);
\fill [white] (6,2) circle (0.1);
\fill [white] (8,0) circle (0.1);
\fill [white] (9,0) circle (0.1);
\fill [white] (9,1) circle (0.1);
\fill [white] (10,0) circle (0.1);
\fill [white] (11,0) circle (0.1);
\fill [white] (11,1) circle (0.1);
\fill [white] (12,0) circle (0.1);

\draw (0,0) circle (0.1);
\draw (1,0) circle (0.1);
\draw (1,1) circle (0.1);
\draw (2,0) circle (0.1);
\draw (3,0) circle (0.1);
\draw (4,0) circle (0.1);
\draw (5,0) circle (0.1);
\draw (5,1) circle (0.1);
\draw (6,0) circle (0.1);
\draw (7,0) circle (0.1);
\draw (7,1) circle (0.1);
\draw (7,2) circle (0.1);
\draw (6,2) circle (0.1);
\draw (8,0) circle (0.1);
\draw (9,0) circle (0.1);
\draw (9,1) circle (0.1);
\draw (10,0) circle (0.1);
\draw (11,0) circle (0.1);
\draw (11,1) circle (0.1);
\draw (12,0) circle (0.1);

\draw [above] (6,2) node{$T$};

\end{tikzpicture}
\end{center}

We compute the double cover defining the Hilbert modular surface.
\begin{theorem}
A birational model for the surface $\widetilde{\sL}_{10}$ (equivalently,
for $Y_{-}(100)$) is given by
\begin{align*}
z^2 &= -(2s^2+2s-r^2)(s^6-2rs^5-r^2s^4+2rs^4+2s^4+4r^3s^3+2r^2s^3  +12rs^3+14s^3 \\
& \qquad \qquad -2r^4s^2-14r^3s^2-29r^2s^2-20rs^2 +10r^4s+10r^3s-2r^2s+r^4).
\end{align*}
Its minimal model is an honestly elliptic surface $X$ with
$\chi(\sO_X) = 3$ and Picard number $28$ or $29$. The Humbert surface
is birational to the $(r,s)$-plane.
\end{theorem}

Making the substitution $r = st$ and absorbing $s^4$ in to $z^2$ makes
the right side quartic in $s$, showing that $\sL_{10}$ is an elliptic
surface with section over $\Proj^1_t$. Converting to the Jacobian
form, after some Weierstrass transformations and substitution $t =
u+1$, we obtain
\begin{align*}
y^2 & = x^3+(u^6+16u^5+74u^4+92u^3+21u^2+20u+4)x^2 \\
   & \qquad +8u^2(10u^5+119u^4+324u^3+291u^2+68u+4)x \\
 &\qquad -16u^4(8u^4+84u^3+139u^2+12u-4).
\end{align*}
It is an honestly elliptic surface with $\chi = 3$, and reducible
fibers of type $\I_{10}$ at $\infty$, $\I_5$ at $0$, $\IV$ at $-2$,
$\I_2$ at the roots of $u^2+3u+1$ and also at the roots of $u^2+8u-4$,
and $\I_3$ at the roots of $3u^3+26u^2+14u+2$. The trivial lattice has
rank $27$, leaving room for Mordell--Weil rank up to $3$.
We find the section
\begin{align*}
P &= \big( -4u^2(3u^2+26u+5), 4u^2(u^2+8u-4)(3u^3+26u^2+14u+2) \big)
\end{align*}
of height $1/5$. Therefore, the Picard number $\rho$ is at least
$28$. On the other hand, counting points modulo $7$ and $17$ shows
that $\rho < 30$. Therefore the Picard number is either $28$ or $29$;
we have not been able to determine it exactly.

The equation of a tautological family of genus $2$ curves and the
elementary symmetric polynomials in the $j$-invariants of the two
elliptic curves are omitted here, but can be found in the computer
files.

\subsection{Special loci}

\begin{enumerate}
\item The rational curves $s^2+2s-r^2 = 0$ (parametrization $r =
  -t(t+2)/\big(2(t+1)\big)$, $s = t^2/\big(2(t+1)\big)$) and
  $s^2+rs+2s-r^2 = 0$ (parametrization $r =
  2(t-3)(3t-4)/\big(5(t^2-t-1)\big)$, $s =
  2(t-3)^2/\big(5(t^2-t-1)\big)$)  are part of the product locus, and
  lift to elliptic curves on $\widetilde{\sL}_{10}$. The rational
  curve $2s^2+2s-r^2 = 0$ (parametrized by $r = -2t/(t^2-2)$, $s =
  2/(t^2-2)$) is also part of the product locus, as well as the branch
  locus.
\item The curves $s = -2$, $s = -r$, $s = 0$, $s = r$ and $s = -r-2$
  are part of the locus $I_{10} = 0$. The first two lift to rational
  curves, while the others lift to unions of two rational curves.
\item The curve $r = 0$ lifts to an elliptic curve isomorphic to
  $X_0(11)$. It corresponds to the two elliptic curves being
  $11$-isogenous.
\item The curve $r = -2$ lifts to an elliptic curve isomorphic to
  $X_0(19)$. It corresponds to the elliptic curves being
  $19$-isogenous. 
\item The curve $s = -1$ also lifts to an elliptic curve (with
  conductor 36 and CM), and is modular: it corresponds to the elliptic
  curves being $9$-isogenous.
\item The curve $r =-1$ lifts to a rational curve; the two elliptic
  curves are again $9$-isogenous along that locus.
\item The branch locus is an elliptic curve, and corresponds to $j_1 =
  j_2$.
\item The section above gives rise to the non-modular rational curve
  parametrized by
\begin{align*}
r &= 2(u+1)(3u+1)/\big(u(3u^2+8u+2)\big), \\
s &= 2(3u+1)/\big(u(3u^2+8u+2)\big), \\
z &= 8(3u+1)^2(u^2+3u+1)(3u^3+26u^2+14u+2)/\big(u^3(3u^2+8u+2)^4\big).
\end{align*}
\end{enumerate}

\section{Discriminant $121$}

\subsection{Parametrization}

Finally, we describe moduli space of elliptic subfields of degree $11$
via the Hilbert modular surface of discriminant $121$. To compute the
surface, we work with the moduli space of elliptic K3 surfaces with
$A_6$, $A_5$ and $A_3$ fibers, and two sections with intersection
matrix
\[
\frac{1}{84} 
\begin{pmatrix}
59 & -46 \\ -46 & 61
\end{pmatrix}
\]
with the entries arising as
\begin{align*}
59/84 &= 4 - 3/4 - 5/6 - 12/7, \\
61/42 &= 4 - 5/6 - 12/7, \\
-23/42 &= 2 - 5/6 - 12/7.
\end{align*}
We will be very brief in our description of all the steps, since the
formulas are all fairly complicated and explained in greater length in
the auxiliary files.

The moduli space of these K3 surfaces is again rational, and with
chosen parameters $r$ and $s$, the universal Weierstrass equation is
as follows.
\begin{align*}
y^2 &= x^3 + a(t) x^2 + b(t) x + c(t), 
\end{align*}
where 
\begin{align*}
a(t) &= (s-r)^2  t^4 -16 \, t^3\, r(rs^4+3s^4-2r^2s^3-5rs^3-4s^3+r^3s^2+4r^2s^2+5rs^2 \\
&\qquad \qquad \qquad \qquad \qquad +2s^2-2r^3s-4r^2s-4rs-s+r^3+2r^2+r)/s   \\
& \qquad + 64r(s-1)^2\, t^2 \,(r^3s^4+6r^2s^4+7rs^4+4s^4-2r^4s^3-8r^3s^3-12r^2s^3-10rs^3 \\
& \qquad \qquad \qquad \qquad \qquad -6s^3+r^5s^2 +4r^4s^2+7r^3s^2+6r^2s^2+5rs^2+2s^2-2r^5s \\
& \qquad \qquad \qquad \qquad \qquad -6r^4s-10r^3s-10r^2s-4rs+r^5+4r^4+6r^3+4r^2+r)/s^2  \\
& \qquad -1024 \, t \, r(s-1)^2(s-r-1)(r^3s^4+5r^2s^4+4rs^4-2s^4-r^4s^3-5r^3s^3-8r^2s^3 \\
& \qquad \qquad \qquad \qquad \qquad \qquad \qquad \quad +rs^3+4s^3 +r^4s^2-8r^2s^2-10rs^2-2s^2 +r^4s\\
& \qquad \qquad \qquad \qquad \qquad \qquad \qquad \quad +7r^3s+12r^2s+6rs-r^4-3r^3-3r^2-r)/s^2  \\
& \qquad  + 4096r^2(s-1)^4(s-r-1)^2(rs+2s-r-1)^2/s^2, \\
b(t) &=  512t^2r(s-1)^2(s-r-1)(s^2-r^2s-2rs-s+r^2+r)\big(t - 8(r^2s+rs-s-r^2+1)\big) \cdot \\
& \qquad \Big( s(s-r)^2  t^3  - 8 \, t^2 \, r(2rs^4+5s^4-3r^2s^3-7rs^3-5s^3+r^3s^2+4r^2s^2 +3rs^2\\
& \qquad \qquad \qquad \qquad \qquad \quad +s^2-2r^3s -3r^2s-3rs-s+r^3+2r^2+r)  \\
& \qquad \quad  + 64r(s-1)^2\, t \,(r^3s^3+5r^2s^3+5rs^3-r^4s^2-4r^3s^2-5r^2s^2+rs^2 +s^2\\
& \qquad \qquad \qquad \qquad \qquad -3r^3s -10r^2s-8rs-s+r^4+4r^3+5r^2+2r)   \\
& \qquad \quad - 512r^2(s-1)^3(s-r-1)(rs+2s-r-1)^2 \Big)/s^2, 
\end{align*}
\begin{align*}
c(t) &= 65536t^4r^2(s-1)^4(s-r-1)^2(s^2-r^2s-2rs-s+r^2+r)^2(t - 8(r^2s+rs-s-r^2+1))^2 \cdot \\
& \qquad  \Big( (s-r)^2  t^2 -16r(rs^3+2s^3-r^2s^2-2rs^2-s^2-2rs-s+r^2+r)t \\
& \qquad \qquad  + 64r^2(s-1)^2(rs+2s-r-1)^2 \Big)/s^2.
\end{align*}
The $x$-coordinates of the two sections are given by 
\begin{align*}
x_1 &=  256 t r(s-1)^2(s-r-1)(8r^2s+8rs-8s-8r^2-t+8)(s^2-r^2s-2rs-s+r^2+r)/s, \\
x_2 &=  -256tr(s-1)^2(s-r-1) \big( s(s^2-r^2s-2rs-s+r^2+r)  t - 8r(rs+2s-r-1)^2(s^2-r-1) \big)/s^2 .
\end{align*}

\subsection{Map to $\A_2$ and equation of $\widetilde{\sL}_{11}$}

We next quickly illustrate the neighbor steps required to go to an
$E_8 E_7$ fibration. 

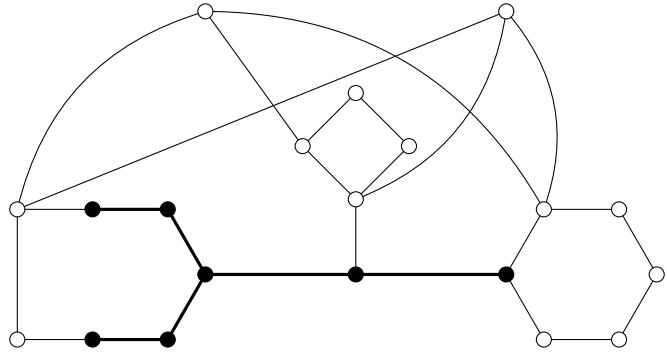
\begin{center}
\begin{tikzpicture}

\draw (-2.5,-0.866)--(-3.5,-0.866)--(-3.5,0.866)--(-2.5,0.866);

\draw [very thick] (-1,0)--(3,0);
\draw (3,0)--(3.5,0.866)--(4.5,0.866)--(5,0)--(4.5,-0.866)--(3.5,-0.866)--(3,0);

\draw [very thick] (-2.5,0.866)--(-1.5,0.866)--(-1,0)--(-1.5,-0.866)--(-2.5,-0.866);

\draw (1,0)--(1,1)--(1.707,1.707)--(1,2.414)--(1-0.707,1.707)--(1,1);

\draw [bend left] (-1,3.5) to (3.5,0.866);
\draw [bend left] (3,3.5) to (3.5,0.866);
\draw [bend right] (-1,3.5) to (-3.5,0.866);
\draw (3,3.5)--(-3.5,0.866);
\draw (-1,3.5)--(1-0.707,1.707);
\draw [bend left] (3,3.5) to (1,1);

\fill [black] (-1,0) circle (0.1);
\fill [black] (1,0) circle (0.1);
\fill [black] (3,0) circle (0.1);
\fill [white] (1,1) circle (0.1);
\fill [white] (1.707,1.707) circle (0.1);
\fill [white] (1,2.414) circle (0.1);
\fill [white] (1-0.707,1.707) circle (0.1);

\fill [white] (3.5,0.866) circle (0.1);
\fill [white] (3.5,-0.866) circle (0.1);
\fill [white] (4.5,0.866) circle (0.1);
\fill [white] (4.5,-0.866) circle (0.1);
\fill [white] (5,0) circle (0.1);

\fill [black] (-1.5,0.866) circle (0.1);
\fill [black] (-1.5,-0.866) circle (0.1);
\fill [black] (-2.5,0.866) circle (0.1);
\fill [black] (-2.5,-0.866) circle (0.1);
\fill [white] (-3.5,0.866) circle (0.1);
\fill [white] (-3.5,-0.866) circle (0.1);

\draw (-1,0) circle (0.1);
\draw (1,0) circle (0.1);
\draw (3,0) circle (0.1);
\draw (1,1) circle (0.1);
\draw (1.707,1.707) circle (0.1);
\draw (1,2.414) circle (0.1);
\draw (1-0.707,1.707) circle (0.1);

\draw (3.5,0.866) circle (0.1);
\draw (3.5,-0.866) circle (0.1);
\draw (4.5,0.866) circle (0.1);
\draw (4.5,-0.866) circle (0.1);
\draw (5,0) circle (0.1);

\draw (-1.5,0.866) circle (0.1);
\draw (-1.5,-0.866) circle (0.1);
\draw (-2.5,0.866) circle (0.1);
\draw (-2.5,-0.866) circle (0.1);
\draw (-3.5,0.866) circle (0.1);
\draw (-3.5,-0.866) circle (0.1);

\fill [white] (-1,3.5) circle (0.1);
\fill [white] (3,3.5) circle (0.1);

\draw (-1,3.5) circle (0.1);
\draw (3,3.5) circle (0.1);

\end{tikzpicture}
\end{center}

First, we take a $2$-neighbor step to an elliptic fibration with
$E_6$, $A_4$, $A_3$ and $A_1$ fibers, and two sections with
intersection matrix
\[
\frac{1}{60} \begin{pmatrix}
37 & 25 \\ 25 & 115
\end{pmatrix},
\]
with the entries being realized as 
\begin{align*}
37/60 &= 4 - 4/3 - 1/2 - 4/5 - 3/4 \\
23/12 &= 4 - 4/3 - 3/4 \\
5/12 &= 2 - 4/3 - 1/4.
\end{align*}
In the figure below we only display the section of height $23/12$, in
order to not clutter the picture.

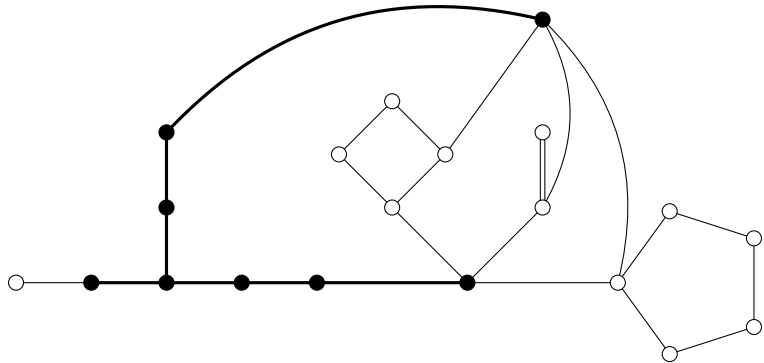
\begin{center}
\begin{tikzpicture}

\draw (0,0)--(1,0);
\draw [very thick] (1,0)--(6,0);
\draw [very thick] (2,0)--(2,2);
\draw (6,0)--(5,1)--(5+0.707,1.707)--(5,2.414)--(5-0.707,1.707)--(5,1);
\draw (6,0)--(7,1);
\draw (6,0)--(8,0)--(8+0.69,0.95)--(8+1.81,0.59)--(8+1.81,-0.59)--(8+0.69,-0.95)--(8,0);
\draw (7.03,1)--(7.03,2);
\draw (6.97,1)--(6.97,2);

\draw [very thick, bend right] (7,3.5) to (2,2);
\draw (7,3.5)--(5+0.707,1.707);
\draw [bend left] (7,3.5) to (7,1);
\draw [bend left] (7,3.5) to (8,0);

\fill [white] (0,0) circle (0.1);
\fill [black] (1,0) circle (0.1);
\fill [black] (2,0) circle (0.1);
\fill [black] (3,0) circle (0.1);
\fill [black] (4,0) circle (0.1);
\fill [black] (6,0) circle (0.1);
\fill [white] (8,0) circle (0.1);
\fill [black] (2,1) circle (0.1);
\fill [black] (2,2) circle (0.1);

\fill [white] (5,1) circle (0.1);
\fill [white] (7,1) circle (0.1);
\fill [white] (7,2) circle (0.1);

\fill [black] (7,3.5) circle (0.1);

\fill [white] (5+0.707,1.707) circle (0.1);
\fill [white] (5-0.707,1.707) circle (0.1);
\fill [white] (5,2.414) circle (0.1);

\fill [white] (8+0.69,0.95) circle (0.1);
\fill [white] (8+0.69,-0.95) circle (0.1);
\fill [white] (8+1.81,0.59) circle (0.1);
\fill [white] (8+1.81,-0.59) circle (0.1);

\draw (0,0) circle (0.1);
\draw (1,0) circle (0.1);
\draw (2,0) circle (0.1);
\draw (3,0) circle (0.1);
\draw (4,0) circle (0.1);
\draw (6,0) circle (0.1);
\draw (8,0) circle (0.1);
\draw (2,1) circle (0.1);
\draw (2,2) circle (0.1);

\draw (5,1) circle (0.1);
\draw (7,1) circle (0.1);
\draw (7,2) circle (0.1);

\draw (7,3.5) circle (0.1);

\draw (5+0.707,1.707) circle (0.1);
\draw (5-0.707,1.707) circle (0.1);
\draw (5,2.414) circle (0.1);

\draw (8+0.69,0.95) circle (0.1);
\draw (8+0.69,-0.95) circle (0.1);
\draw (8+1.81,0.59) circle (0.1);
\draw (8+1.81,-0.59) circle (0.1);

\end{tikzpicture}
\end{center}

Next, we go by a $2$-neighbor step to an $E_7 A_4 A_2 A_1$ fibration
with two sections having intersection matrix
\[
\frac{1}{30}\begin{pmatrix}
16 & -5 \\ -5 & 115
\end{pmatrix},
\]
with the entries being realized as
\begin{align*}
8/15 &= 4 - 3/2 - 4/5 - 2/3 - 1/2 \\
23/6 &= 6 - 3/2 - 2/3 \\
-1/6 &= 2 + 1 - 1 - 3/2 - 2/3.
\end{align*}
We only show the section of height $8/15$ below.

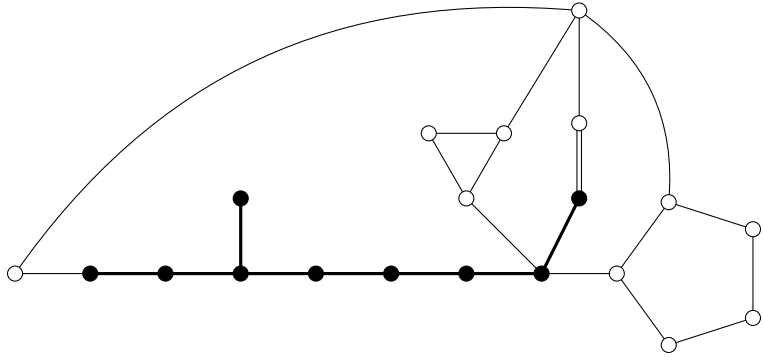
\begin{center}
\begin{tikzpicture}

\draw (-1,0)--(0,0);
\draw [very thick] (0,0)--(6,0);
\draw [very thick] (2,0)--(2,1);
\draw (6,0)--(5,1)--(5+0.5,1.866)--(5-0.5,1.866)--(5,1);
\draw [very thick] (6,0)--(6.5,1);
\draw (6,0)--(7,0)--(7+0.69,0.95)--(7+1.81,0.59)--(7+1.81,-0.59)--(7+0.69,-0.95)--(7,0);
\draw (6.53,1)--(6.53,2);
\draw (6.47,1)--(6.47,2);

\draw [bend right] (6.5,3.5) to (-1,0);
\draw (6.5,3.5)--(5+0.5,1.866);
\draw (6.5,3.5)--(6.5,2);
\draw [bend left] (6.5,3.5) to (7+0.69,0.95);

\fill [white] (-1,0) circle (0.1);
\fill [black] (0,0) circle (0.1);
\fill [black] (1,0) circle (0.1);
\fill [black] (2,0) circle (0.1);
\fill [black] (3,0) circle (0.1);
\fill [black] (4,0) circle (0.1);
\fill [black] (5,0) circle (0.1);
\fill [black] (6,0) circle (0.1);
\fill [white] (7,0) circle (0.1);
\fill [black] (2,1) circle (0.1);

\fill [white] (5,1) circle (0.1);
\fill [black] (6.5,1) circle (0.1);
\fill [white] (6.5,2) circle (0.1);

\fill [white] (6.5,3.5) circle (0.1);

\fill [white] (5+0.5,1.866) circle (0.1);
\fill [white] (5-0.5,1.866) circle (0.1);

\fill [white] (7+0.69,0.95) circle (0.1);
\fill [white] (7+0.69,-0.95) circle (0.1);
\fill [white] (7+1.81,0.59) circle (0.1);
\fill [white] (7+1.81,-0.59) circle (0.1);

\draw (-1,0) circle (0.1);
\draw (0,0) circle (0.1);
\draw (1,0) circle (0.1);
\draw (2,0) circle (0.1);
\draw (3,0) circle (0.1);
\draw (4,0) circle (0.1);
\draw (5,0) circle (0.1);
\draw (6,0) circle (0.1);
\draw (7,0) circle (0.1);
\draw (2,1) circle (0.1);

\draw (5,1) circle (0.1);
\draw (6.5,1) circle (0.1);
\draw (6.5,2) circle (0.1);

\draw (6.5,3.5) circle (0.1);

\draw (5+0.5,1.866) circle (0.1);
\draw (5-0.5,1.866) circle (0.1);

\draw (7+0.69,0.95) circle (0.1);
\draw (7+0.69,-0.95) circle (0.1);
\draw (7+1.81,0.59) circle (0.1);
\draw (7+1.81,-0.59) circle (0.1);

\end{tikzpicture}
\end{center}

We next take a $2$-neighbor step to an $E_8 A_7$ fibration, with a
section of height $121/8 = 4 + 2 \cdot 6 - 7/8$. Finally, we take a
$2$-neighbor step to go to an $E_8 E_7$ fibration, as in the second
step for discriminant $25$. The formulas for the Igusa--Clebsch
invariants can now be read off; they may be found in the auxiliary
files.

Finally, we compute the double cover defining the Hilbert modular surface.
\begin{theorem}
A birational model for the surface $\widetilde{\sL}_{11}$ (equivalently,
for $Y_{-}(121)$) is given by
\begin{align*}
z^2 &= (s-1)^4  r^6 -2(s-1)^3(s^2+4s+2)r^5 +(s-1)^2(s^4-4s^3-40s^2-10s+6) r^4 \\
& \qquad + 2(s-1)(6s^5-3s^4-25s^3-22s^2+22s-2) r^3 + (22s^6-8s^5-10s^4-108s^3+113s^2-26s+1) r^2 \\
& \qquad -2(s-1)s^2(s+1)(2s-13)(2s-1) r -27(s-1)^2s^4.
\end{align*}
It is a surface of general type. The Humbert surface is birational to
the $(r,s)$-plane.
\end{theorem}

The expressions for the sum and product of the $j$-invariants of the
two associated elliptic curves may be found in the auxiliary computer
files, which also have the equation of a tautological family of genus
$2$ curves.

\subsection{Special loci}
Once again we describe some special curves of low genus on the modular
surface.
\begin{enumerate}
\item The rational curve $s = 0$ and the genus $1$ curve
  $s^2-r^2s-2rs-s+r^2+r = 0$ are contained in the product locus. Each
  lifts to a union of two curves on the Hilbert modular surface (of
  genera $0$ and $1$ respectively).
\item The curves
\begin{align*}
& r = 0 \\
& s = 1 \\
& s = r + 1 \\
& rs+s-r = 0 \\
& r^2s+rs-s-r^2+1  = 0 \\
& s^2-r^2s-rs-s+r^2 = 0 \\
& rs^2+2s^2-r^2s-3rs-2s+r^2+r= 0 
\end{align*}
all correspond to $I_{10} = 0$. The first five are of genus $0$ and
the remaining two have genus $1$. The first two lift to unions of two
genus $0$ curves with no rational points, the third to a union of two
$\Proj^1$'s, the fourth to an irreducible $\Proj^1$, the fifth to an
elliptic curve, and the last two to curves of genus $2$ and $3$
respectively.
\item The rational curve $rs + 2s - r -1 = 0$ lifts to an elliptic
  curve isomorphic to $X_0(21)$, and corresponds to the elliptic
  curves being $21$-isogenous.
\item the rational curve $r = -1$ lifts to an elliptic curve
  isomorphic to $X_0(17)$, and corresponds to the elliptic curves
  being $17$-isogenous.
\item The rational curve $s^2 + rs -r = 0$ lifts to an elliptic curve
  isomorphic to $X_0(32)$, and corresponds to the the elliptic curves
  being $32$-isogenous.
\end{enumerate}

\end{document}